\RequirePackage{fix-cm}
\documentclass[smallextended]{svjour3}       
\smartqed  
\usepackage{subcaption}

\usepackage{amsmath,amssymb}
\newcommand{\bsym}[1]{\ensuremath{\boldsymbol{#1}}}
\renewcommand{\vec}[1]{\ensuremath{\bsym{#1}}}
\newcommand{\fnc}[1]{\ensuremath{\mathcal{#1}}}
\newcommand{\mat}[1]{\ensuremath{\mathsf{#1}}}
\newcommand{\norm}[1]{\left\lVert #1 \right\rVert}

\newcommand{\der}[2][]{\dfrac{\mathrm{d} #1}{\mathrm{d} #2}}
\usepackage{mathtools}  
\newcommand{\diag}[1]{\operatorname{diag}\left(#1\right)}

\usepackage{xcolor}
\usepackage{enumitem}
\usepackage{hyperref} 
\hypersetup{
    colorlinks,
    linkcolor={black},
    citecolor={black},
    urlcolor={blue!80!black}
}

\begin{document}

\title{Construction and Optimization of Summation-by-Parts Operators for General Function Spaces Using an Improved Generalized Gaussian Quadrature Algorithm}

\titlerunning{Construction and Optimization of FSBP Operators Using Gaussian Quadrature} 

\author{Alex Bercik         \and
        Lisa Patrascu \and
        David W. Zingg 
}


\institute{A.\ Bercik \at
              University of Toronto Institute for Aerospace Studies, Toronto, Canada \\
              \email{alex.bercik@mail.utoronto.ca}
            \and
              L.\ Patrascu \at
                 University of Toronto Institute for Aerospace Studies, Toronto, Canada \\
                 \email{lisa.patrascu@mail.utoronto.ca}
            \and
           D.\ W.\ Zingg \at
              University of Toronto Institute for Aerospace Studies, Toronto, Canada \\
              \email{david.zingg@utoronto.ca}
}

\date{Received: date / Accepted: date}

\maketitle

\begin{abstract}

We construct optimized summation-by-parts (SBP) operators for general function spaces with provably minimal degrees of freedom on open, closed, and half-open nodal distributions. These operators rely on generalized Gaussian quadrature rules, for which we present an improved algorithm that is flexible, efficient, and provably convergent. In cases where free parameters are available, we further introduce two operator optimization strategies. We test our operators on a handful of numerical examples that contain large or unbounded gradients, in which some a priori knowledge of the solution has been assumed to select an appropriate basis. The novel operators are found to outperform standard polynomial operators by several orders of magnitude in solution accuracy relative to degrees of freedom. Furthermore, our novel operators significantly outperform function-space SBP operators with equispaced nodal distributions, which require significantly more nodes for the same operator basis. Finally, we demonstrate that the operator optimization procedures are critical to achieving accurate and efficient discretizations, as the standard SBP construction procedure can lead to nullspace-inconsistent and poorly-conditioned operators.

\end{abstract}

\section{Introduction}\label{sec:introduction}

Summation-by-parts (SBP) operators, together with simultaneous approximation terms (SATs) for enforcing boundary conditions, allow for the construction of provably robust high-order discretizations of partial differential equations. By satisfying integration by parts at the semidiscrete level, they provide an algebraic framework for reproducing stability and conservation proofs of the continuous problem~\cite{Kreiss1974,strand1994,carpenter1999,DelReyFernandez2014_review,svard_review}. As with any numerical method that achieves high-order through polynomial approximation, however, accuracy can suffer when applied to practical problems of interest involving singularities, high gradients, or other features requiring high resolution, such as cracks in solid mechanics~\cite{Byskov1970,Fix1973,Benzley1974}, or boundary layers and singularities in fluid mechanics~\cite{yano_airfoil,krank_wall_2018,brill_enriched_2020,zhang_exponential_2020,Fidkowski_basis_2025}.

The use of nonpolynomial basis functions to improve the accuracy of a numerical method without remeshing has been studied in the context of the finite element method as early as the 1970s~\cite{Byskov1970,Fix1973,Benzley1974}. Several closely related methods have since emerged, including the Extended/Generalized Finite Element Methods (XFEM/GFEM) \cite{XFEM_1,XFEM_2,fries_extendedgeneralized_2010}, the Discontinuous Enrichment Method (DEM) \cite{farhat_discontinuous_2001}, and the Discontinuous Galerkin (DG) method with non-polynomial approximation spaces~\cite{yuan_discontinuous_2006}. The XFEM/GFEM can be classified under the broader definition of a partition of unity method~\cite{Babuska_PUM_1994,Melenk_PUM_1996,Babuska_PUM_1997}, where the basis is augmented by multiplications of nonpolynomial `enrichment functions' with the partition of unity basis shape functions. This allows discontinuities and singularities to be more effectively represented while maintaining the compact support of the local basis functions. For a comprehensive overview of these methods, we refer to~\cite{fries_extendedgeneralized_2010}. In contrast, the DEM (and variants thereof, e.g. \cite{borker_discontinuous_2017}) directly augments the basis with nonpolynomial functions and enforces inter-element continuity weakly via Lagrange multipliers. While both methods allow for great flexibility in terms of basis enrichment, it is not yet currently understood whether these can be incorporated into an SBP framework. Challenges include the treatment of element interfaces and the exact integration of the enriched basis. On the other hand, standard DG methods (which are interpretable as SBP methods under the choice of specific quadratures~\cite{gassner2013}) can be directly implemented with certain nonpolynomial bases and retain design-order accuracy~\cite{yuan_discontinuous_2006,Yang_dg_2016}. The theory for analogously using nonpolynomial basis functions within the broader SBP framework was developed in~\cite{glaubitz_fsbp2023}, with later extensions to multidimensional operators in~\cite{glaubitz_multidimensional_2023} and the second derivative in~\cite{glaubitz_second_2024}. In these works, the quadrature rules are derived directly from optimization procedures (e.g.~\cite{glaubitz_quadrature_2023,glaubitz_multidimensional_2023,Glaubitz_optimization_2025}).

Until recently, these general function-space SBP (FSBP) operators typically employed more degrees of freedom (DOF) than necessary due to their use of uniformly-spaced nodes. Therefore, although the operators were shown to exhibit improved performance over traditional polynomial-based SBP operators for certain problems, their efficiency is likely not optimal. Recently, Hale et al.~\cite{hale_sbp_2026} constructed FSBP operators with provably minimal DOF by employing generalized Gaussian quadratures, labelling them `optimal'. While the use of Gaussian quadratures to construct SBP operators with minimal DOF is well-known for polynomial-based operators~\cite{DDRF_generalized_2014}, applying the same theory to general function spaces introduces challenges. For example, while polynomial Gauss quadrature can easily be obtained through the roots of orthogonal polynomials, there is no comparable characterization for general systems. At best, one can rely on the rich theory of Chebyshev systems to characterize the existence of quadrature rules (e.g. see~\cite{KarlinStudden_1966}), but to actually find the rules one must still employ nonlinear algorithms that are not always guaranteed to converge (e.g.~\cite{ma_generalized_1994,yarvin_generalized_1998,bremer_nonlinear_2010}).

In this work we improve the generalized Gaussian quadrature algorithm of Huybrechs~\cite{Huybrechs_computation_2022} to reliably and efficiently construct quadrature rules suitable for FSBP operators with minimal DOF. This algorithm has the particular advantages that it is guaranteed to converge, and is remarkably flexible, suitable for finding quadrature rules with closed, open, and half-open nodal distributions with arbitrary positive measure. This enables the construction of FSBP operators on general nodal distributions, including open and half-open operators that are suitable for the simulation of problems with endpoint singularities. For cases where the quadrature rule results in more nodes than are needed to represent the basis functions exactly, we also introduce two methodologies to exploit these additional DOF and obtain optimized FSBP operators. In \S \ref{sec:fsbp-operators} we briefly review the theory of FSBP operators, then similarly review the theory of generalized Gaussian quadratures in \S \ref{sec:gen_gauss}. In \S \ref{sec:gen_gauss_algo} we describe the improved generalized Gauss algorithm employed in this work, and in \S \ref{sec:optimization} we introduce the two FSBP operator optimization strategies. Finally, in \S \ref{sec:results} we verify our novel optimized operators on various simple test problems, including one with an endpoint singularity.

\section{FSBP operators} \label{sec:fsbp-operators}

For a comprehensive introduction to FSBP operators, we refer the reader to \cite{glaubitz_fsbp2023}. Here we briefly review the key concepts. We begin by generalizing Definition 2.3 of \cite{glaubitz_fsbp2023} to allow for FSBP operators with arbitrary nodal distributions and possibly endpoint singularities, analogous to the generalization introduced in \cite{DDRF_generalized_2014} for classical polynomial-based SBP operators.
\begin{definition} \label{defn:fsbp-operator}
  Let $\fnc{F} \subset C^0([x_L,x_R]) \cap C^1((x_L,x_R))$ be a finite-dimensional function space. An operator $\mat{D} = \mat{H}^{-1} \mat{Q} \in \mathbb{R}^{N \times N}$ defined on a nodal distribution $\vec{x} \coloneqq \left\{ x_i \mid x_L \leq x_1 < \ldots < x_N \leq x_R \right\}$ is an $\fnc{F}$-based SBP operator if
  \begin{enumerate}[label=\roman*)]
    \item $\mat{D} f(\vec{x}) = f'(\vec{x})$ for all $f \in \fnc{F}$,
    \item $\mat{H}$ is a symmetric positive definite matrix, and
    \item $\mat{Q} + \mat{Q}^\top = \vec{t}_R \vec{t}_R^\top - \vec{t}_L \vec{t}_L^\top$, where $\vec{t}_R^\top f(\vec{x}) = f(x_R)$ and $\vec{t}_L^\top f(\vec{x}) = f(x_L)$ for all $f \in \fnc{F}$.
  \end{enumerate}
\end{definition}
We denote the boundary integration matrix by $\mat{E} \coloneqq \vec{t}_R \vec{t}_R^\top - \vec{t}_L \vec{t}_L^\top$, where if the endpoints are included in the nodal distribution, i.e.\,$x_1 = x_L$ and $x_N = x_R$, we can set the extrapolation operators as unit vectors, i.e.\,$\vec{t}_R = \vec{e}_N$ and $\vec{t}_L = \vec{e}_1$, such that we recover the definition found in \cite{glaubitz_fsbp2023} with $\mat{E} = \diag{-1, 0,  \ldots, 1}$. In this work we exclusively consider a diagonal matrix $\mat{H}$, such that its entries correspond to a positive quadrature rule, as defined below.

\begin{definition}[Definition 4.1 of \cite{glaubitz_fsbp2023}]
  An $N$-point quadrature rule $(\vec{x}, \vec{\omega})$ of the form $\sum_{i=1}^{N} \omega_i f(x_i)$ is positive if $\omega_i > 0$ for all $i = 1, \ldots, N$, and $\fnc{G}$-exact if it integrates all functions in the function space $\fnc{G}$ exactly, i.e.
  \begin{equation}
    \sum_{i=1}^{N} \omega_i g(x_i)  = \int_{x_L}^{x_R} g(x) \, \mathrm{d}x
    \quad \text{for all } g \in \fnc{G}.
  \end{equation}
\end{definition}
The existence of a positive quadrature rule that is exact on the function space $(\fnc{F} \fnc{F})' \coloneqq \left\{ f'g + g' f \mid f, g \in \fnc{F} \right\}$ guarantees the existence of a diagonal-$\mat{H}$ FSBP operator for the space $\fnc{F}$, as is made clear in the following Theorem.

\begin{theorem}[Corollary 4.6 of \cite{glaubitz_fsbp2023}] \label{thm:existence}
  Let $\fnc{F} \subset C^1([x_L,x_R])$ be a finite-dimensional function space with basis $\{ f_i \}_{i=1}^K$, and assume that the generalized Vandermonde matrix $\mat{V}_{ij} \coloneqq f_j (x_i)$ has linearly independent columns. Then there exists an $\fnc{F}$-based SBP operator $\mat{D}$ with a positive diagonal $\mat{H}$ if and only if there exists a positive and $(\fnc{F} \fnc{F})'$-exact quadrature rule defined on the same nodal distribution $\vec{x}$.
\end{theorem}
Given a general function space $\fnc{F}$, Theorem~\ref{thm:existence} allows the construction of an $\fnc{F}$-based SBP operator to be reduced to a two-step procedure in which one first finds a positive and $(\fnc{F} \fnc{F})'$-exact quadrature formula and then constructs the SBP operator by enforcing the accuracy and SBP conditions of Definition~\ref{defn:fsbp-operator} (i and iii, respectively). This two-step procedure is commonplace in both the polynomial and general function space SBP literature \cite{DDRF_generalized_2014,Marchildon_optimization_2020,glaubitz_fsbp2023,Worku_quadrature_2024}, though it is also possible to construct the quadrature rule and operator simultaneously \cite{Mattsson_optimal_2014,Glaubitz_optimization_2025}.

\section{Generalized Gaussian Quadrature} \label{sec:gen_gauss}

Polynomial-based diagonal-$\mat{H}$ SBP operators with the minimal possible number of nodes typically employ Gaussian quadrature~\cite{DDRF_generalized_2014}; either Legendre-Gauss (LG) with an open nodal distribution, i.e.\,no endpoints, Legendre-Gauss-Lobatto (LGL) with a closed nodal distribution, i.e.\,endpoints included, or Gauss-Radau (GR) with a half-open nodal distribution, i.e.\,only one of either the right or left endpoint included. For non-polynomial function spaces, we can similarly employ \emph{generalized} Gaussian quadrature, for which the natural generalization of a set of orthogonal polynomials is a Chebyshev system (or T-system\footnote{We say T-system due to the alternative spelling (e.g. used in \cite{KarlinStudden_1966}), Tchebysheff system.})~\cite{KarlinStudden_1966}, which we define below.

\begin{definition}[Definition 1.1 of {\cite[\S1.1]{KarlinStudden_1966}}] \label{defn:t-system}
  A set $\left\{ f_i \right\}_{i=1}^K$ is a T-system on $[x_L,x_R]$ if $f_j \in C^0([x_L,x_R])$ and the determinant of the generalized Vandermonde matrix $\mat{V}_{ij} \coloneqq f_j (x_i)$ is strictly positive for any nodal distribution $\vec{x} \coloneqq \left\{ x_i \mid x_L \leq x_1 < \ldots < x_K \leq x_R \right\}$. It is a complete Chebyshev system (or CT-system) if every subset $\left\{ f_i \right\}_{i=1}^n$ for $n < K$ is also a T-system.
\end{definition}

For an excellent review of T-systems in relation to generalized Gaussian quadrature, we refer the reader to \cite{Huybrechs_computation_2022}. For a more comprehensive treatment, we instead refer to the foundational work of Karlin and Studden~\cite{KarlinStudden_1966}. Here we summarize only the relevant results. We first define the moment vector $\vec{c} \in \mathbb{R}^K$, where
\begin{equation} \label{eq:moment_vec}
  c_i = \int_{x_L}^{x_R} f_i(x) \, \mathrm{d} \mu(x)
\end{equation}
for some positive measure $\mu$ on $[x_L,x_R]$. Crucially, for a given set of functions $\left\{ f_i \right\}_{i=1}^K$, different measures can result in the same moment vector $\vec{c}$. For example, for the set of degree $2n$ polynomials on $[-1,1]$, the uniform continuous measure $\mathrm{d} \mu(x) = \mathrm{d} x$ produces the same moment vector as the discrete measure $\mathrm{d} \mu(x) = \sum_{i=1}^{n} \omega_i \delta(x - x_i)$, where $(x_i, \omega_i)$ correspond to the $n$ LG quadrature nodes and weights, and $\delta$ is the Dirac delta function. More generally, the well-known Gauss-Jacobi quadratures produce the same moment vector as $\mathrm{d} \mu(x) = (1-x)^\alpha (1+x)^\beta \mathrm{d} x$ for any $\alpha, \beta > -1$. Therefore, given some function set and continuous measure $\mathrm{d} \mu(x) = \omega(x) \mathrm{d} x$ with $\omega(x) > 0$, we seek a \emph{convex representation} of the moment vector $\vec{c}$, i.e.
\begin{equation} \label{eq:convex_representation}
  \vec{c} = \sum_{i=1}^{n} \omega_i \vec{f}(x_i) , \quad \vec{f} \coloneqq \left( f_1, \ldots, f_K \right)^\top.
\end{equation}
for some nodes $x_i \in [x_L,x_R]$ and weights $\omega_i \geq 0$. When a convex representation contains the minimal possible number of $n$ nodes (under appropriate endpoint conditions), it will be our desired quadrature rule. We characterize the existence of minimal-node convex representations in the following theorem.

\begin{theorem}
  \label{thm:generalized_gaussian_quadrature}
  Let $\{f_i\}_{i=1}^K$ be a T-system on $[x_L,x_R]$, and let
  \[
    \mathrm{d}\mu(x) = \omega(x)\,\mathrm{d}x,
    \qquad
    \omega\in C^0([x_L,x_R]),
    \qquad
    \omega(x)>0
    \quad \text{for all } x\in [x_L,x_R].
  \]
  Then the moment vector $\vec{c}$ defined in \eqref{eq:moment_vec} admits the following positive convex representations:
\begin{itemize}
\item If $K$ is even, there exists a unique convex representation \eqref{eq:convex_representation} with $n=K/2$ strictly interior nodes,
\[
(\vec{x}, \vec{\omega}) : x_L < x_1 < \cdots < x_{K/2} < x_R, \quad \omega_i > 0.
\]
This representation has the smallest possible number of nodes among all positive convex representations of $\vec{c}$.
\item If $K$ is even, there exists a unique convex representation \eqref{eq:convex_representation} with $n=K/2+1$ nodes including both endpoints,
\[
(\vec{x}, \vec{\omega}) : x_L = x_1 < \cdots < x_{K/2+1} = x_R, \quad \omega_i > 0.
\]
This representation has the smallest possible number of nodes among all positive convex representations of $\vec{c}$ that include both endpoints.
\item If $K$ is odd, there exists a unique convex representation \eqref{eq:convex_representation} with $n=(K+1)/2$ nodes including either the left endpoint but not the right endpoint, or the right endpoint but not the left endpoint,
\[
(\vec{x}, \vec{\omega}) : 
\begin{cases}
x_L = x_1 < \cdots < x_{(K+1)/2} < x_R, \quad \\
x_L < x_1 < \cdots < x_{(K+1)/2} = x_R,
\end{cases}
\quad \omega_i > 0.
\]
These representations have the smallest possible number of nodes among all positive convex representations of $\vec{c}$ that include one endpoint.
\end{itemize}  
\end{theorem}
\begin{proof}
  Let $\mathcal{M}_K \coloneqq \left\{ \vec{c} = \int_{x_L}^{x_R} \vec{f}(x) \, \mathrm{d} \sigma(x) \in \mathbb{R}^K \mid \forall \ \text{positive measures } \mathrm{d} \sigma(x) \right\}$ be the moment space associated with the T-system $\{f_i\}_{i=1}^K$. By Theorem 2.1 of {\cite[\S2.2]{KarlinStudden_1966}}, boundary points of $\mathcal{M}_K$ can only be (uniquely) represented in the convex form \eqref{eq:convex_representation}. Since the measure $\mathrm{d} \mu(x) = \omega(x) \, \mathrm{d} x$ is continuous---thus not of the form \eqref{eq:convex_representation}---$\vec{c}$ is an interior point of $\mathcal{M}_K$. The desired result then follows directly from Corollary 3.1 of {\cite[\S2.3]{KarlinStudden_1966}} after translating their notation of representation indices to the stated node counts in the present work.
\end{proof}

For polynomial bases, the first case in Theorem \ref{thm:generalized_gaussian_quadrature} corresponds to the LG quadrature rule, the second to the LGL rule, and the third to the GR rules. The key to constructing SBP operators on general function spaces with minimal DOF is to therefore ensure that the quadrature function space $(\fnc{F} \fnc{F})'$ forms a T-system. In Appendix \ref{app:chebyshev_tests}, we describe some useful tests (both necessary and sufficient diagnostics) one may employ to check for the T-system property, for which it is necessary to introduce the following definition.

\begin{definition}[Definition 2.4 of {\cite[\S1.2]{KarlinStudden_1966}}] \label{defn:et-system}
  A set $\left\{ f_i \right\}_{i=1}^K$ is an extended Chebyshev system (or ET-system) on $[x_L,x_R]$ if $f_j \in C^{K-1}([x_L,x_R])$ and the determinant of the confluent generalized Vandermonde matrix $\tilde{\mat{V}}_{ij} \coloneqq f_j^{(k)} (x_i)$, where $k$ ranges over the multiplicity of $x_i$, is strictly positive for any (possibly coincident) nodal distribution $\vec{x} \coloneqq \left\{ x_i \mid x_L \leq x_1 \leq \ldots \leq x_K \leq x_R \right\}$. Note that $\tilde{\mat{V}}$ is obtained by replacing the rows of the generalized Vandermonde matrix~$\mat{V}$ corresponding to coincident nodes by successive derivatives of the basis functions $f_j$. In the limiting case where all nodes are coincident, the determinant reduces to the Wronskian. The set is an extended complete Chebyshev system (ECT-system) if every subset $\left\{ f_i \right\}_{i=1}^j$ for $j < K$ is also an ET-system.
\end{definition}

\section{An Improved Algorithm for Finding Generalized Gaussian Quadratures} \label{sec:gen_gauss_algo}

The first numerical method to find generalized Gaussian quadratures was described by Ma, Rokhlin, and Wandzura in~\cite{ma_generalized_1994}. It solves the nonlinear moment equations using a Newton iteration algorithm with initial guesses obtained by a continuation from a simpler known system, such as the polynomial basis and its Gaussian quadrature. A significant disadvantage of this method is that it requires the function space to form an extended Hermite system along the entire continuation path, a strong assumption that can fail for problems of interest~\cite{yarvin_generalized_1998}. This limitation was addressed by Yarvin and Rokhlin in~\cite{yarvin_generalized_1998}, who reformulated the continuation in the integration measure rather than the function space itself. The improved algorithm was again modified by Hale et al. in~\cite{hale_sbp_2026} to find quadratures with closed nodal distributions. Its input is an even number of $K$ quadrature basis functions, with the restriction that they must form an ET-system. Integrals and derivatives of the basis functions are accurately approximated using piecewise polynomial interpolants via the Chebfun package.

In the present work, we instead employ the approach of Huybrechs in~\cite{Huybrechs_computation_2022}. Unlike the preceding Newton/optimization-based methods, this algorithm directly parametrizes the moment-space, iteratively adding basis functions and quadrature nodes to explore consecutively higher-dimensional moment spaces. By continuously deforming the convex representation \eqref{eq:convex_representation} through ($\vec{x}$, $\vec{\omega}$), a continuous path towards the desired moment vector \eqref{eq:moment_vec} is constructed. Each step along this path reduces to the solution of a nonlinear system, but with an arbitrarily close initial guess from the previous step, so the algorithm is guaranteed to converge. This approach has the further advantage of requiring no smoothness on the quadrature basis functions beyond continuity on the open interval. It does, however, require the quadrature basis to form a CT-system. Of particular interest for this work is that this method can naturally produce quadratures with both closed and open nodal distributions for an even number of basis functions, as well as quadratures with half-open nodal distributions for an odd number of basis functions. This will be important for constructing FSBP operators with endpoint singularities.

One of the contributions of this work is the modification of the Gaussian quadrature algorithm of Huybrechs~\cite{Huybrechs_computation_2022}, which we provide as a publicly available package at \url{https://github.com/alexbercik/GeneralizedGauss.jl}. The underlying theory requires substantial background on principal and canonical representations of moment spaces. Consequently, further discussion on the algorithm, as well as our modifications, is deferred to Appendix~\ref{app:gauss_algo}. This package extends the functionality of the original implementation considerably. Whereas the publicly available code accompanying~\cite{Huybrechs_computation_2022} was restricted to open quadrature rules associated with an even number of basis functions, the present package supports the full range of principal representations described by the theory. In particular, it supports both even and odd-dimensional function spaces, open, closed, and half-open quadrature rules, non-differentiable basis functions, arbitrary-precision arithmetic, arbitrary positive measures, endpoint singularities at either boundary, and automatic orthogonalization of basis functions. In addition, several implementation modifications yield significant improvements in computational efficiency. The resulting package provides a practical and flexible tool for constructing generalized Gaussian quadrature rules for CT-systems.

\begin{remark}
  The above discussion on existing algorithms for computing generalized Gaussian quadratures is not exhaustive. For example, Bremer et al.~\cite{bremer_nonlinear_2010} proposed a nonlinear optimization procedure that applies to any square-integrable function space, removing the requirement that the function space form a CT-system. However, it lacks the guaranteed convergence of the algorithm presented above, and may return only approximate quadrature rules for a given function space if no suitable rule is found. Since the SBP property requires quadrature exactness, we do not consider this method further.
\end{remark}

\section{The Construction and Optimization of FSBP operators} \label{sec:optimization}

\subsection{The Standard Construction Procedure} \label{sec:std_construction}
Once a nodal distribution and the quadrature weights are determined, the typical construction of an FSBP operator (e.g. \cite{glaubitz_fsbp2023,hale_sbp_2026}) follows \cite{hicken_multidimensional_2016}. Consider a function space $\fnc{F} \coloneqq \operatorname*{span}{\left\{ f_j \right\}_{j=1}^d}$ and a quadrature rule $\left( \vec{x} , \vec{\omega} \right)$ with $N$ nodes. We begin by constructing the generalized Vandermonde matrix and its derivatives,
\[
\mat{V}_{ij} \coloneqq f_j (x_i) \, , \qquad 
\mat{V}'_{ij} \coloneqq f_j' (x_i) \, , \qquad 
\mat{V}, \mat{V}' \in \mathbb{R}^{N \times d} \, .
\]
$\mat{V}$ has full column rank since $\left\{ f_j \right\}$ is a T-system. We also define the generalized boundary Vandermonde vectors,
\[
\left( v_L \right)_j \coloneqq f_j(x_L) \, , \qquad 
\left( v_R \right)_j \coloneqq f_j(x_R) \, , \qquad
\vec{v}_L , \vec{v}_R \in \mathbb{R}^d \, .
\]

The standard construction procedure is as follows. First the extrapolation operators are computed to satisfy 
\begin{equation} \label{eq:tL_tR_accuracy}
\mat{V}^\top \vec{t}_L = \vec{v}_L \, , \qquad 
\mat{V}^\top \vec{t}_R = \vec{v}_R \, .
\end{equation}
In \cite{glaubitz_fsbp2023,hale_sbp_2026}, the endpoints are included so that the natural choice is $\vec{t}_L = \vec{e}_1$ and $\vec{t}_R = \vec{e}_N$. To handle cases where $N>d$, i.e.\,the extrapolation operators are not unique, for nodal distributions not containing the endpoint we find $\vec{t}_L$ by solving the minimum norm constrained problem
\begin{equation} \label{eq:tL_min_norm}
  \vec{t}_L \coloneqq  \operatorname*{argmin}_{\vec{t} \in \mathbb{R}^{N}} \norm{\vec{t}}^2_\star \quad
  \text{subject to} \quad
  \mat{V}^\top \vec{t}_L = \vec{v}_L \, ,
\end{equation}
and similarly for $\vec{t}_R$, where $\star$ indicates some weighted norm defined by $\norm{\vec{t}}^2_\star \coloneqq \vec{t}^\top \mat{\Theta}_\star \vec{t}$ for some positive definite $\mat{\Theta}_\star \in \mathbb{R}^{N \times N}$. Note that the formal solution to \eqref{eq:tL_min_norm} is given by the Moore-Penrose pseudoinverse,
\[
  \vec{t}_L = \mat{\Theta}_\star^{-1} \mat{V} \left( \mat{V}^\top \mat{\Theta}_\star^{-1} \mat{V} \right)^{-1} \vec{v}_L \, .
\]
The simple choice $\mat{\Theta}_\star = \mat{I}$ recovers the procedure proposed by \cite{DelReyFernandez_simultaneous_2018}, though in this work we use $\mat{\Theta}_\star = \mat{H}^{-1}$ given that the extrapolation operators are weighted by $\mat{H}^{-1}$ in the SATs, and in practice, we find that this results in better-performing operators. Once $\vec{t}_L$ and $\vec{t}_R$ are computed, $\mat{E} \coloneqq \vec{t}_R \vec{t}_R^\top - \vec{t}_L \vec{t}_L^\top $ is constructed and the accuracy conditions of Definition \ref{defn:fsbp-operator} are recast using $\mat{S}$, the skew-symmetric part of $\mat{Q}$:
\begin{equation} \label{eq:S_accuracy}
\mat{S} \mat{V} = \mat{H} \mat{V}' - \tfrac{1}{2} \mat{E} \mat{V} \, , \qquad \mat{Q} = \mat{S} + \tfrac{1}{2} \mat{E} \, .
\end{equation}
The advantage of this approach is that it naturally enforces the SBP property by construction. The system \eqref{eq:S_accuracy} can then be recast as a linear system acting on the strictly lower triangular entries of $\mat{S}$, resulting in $Nd$ equations and $N(N-1)/2$ unknowns. The existence of a solution is guaranteed by Theorem~\ref{thm:existence}, and in practice, \eqref{eq:S_accuracy} often has infinite solutions. In \cite{glaubitz_fsbp2023,hale_sbp_2026}, the unique least-squares solution with minimal Euclidean norm is selected for computational convenience.

\subsection{The Need for Operator Optimization} \label{sec:optimization_need}
When considering non-polynomial function spaces, the quadrature rule often demands more nodes than the accuracy of the SBP operator itself. For example, the exponential-enriched polynomial space from \cite[\S 6.2]{glaubitz_fsbp2023},
\[
\fnc{F}_p \coloneqq \fnc{P}_{p} + \operatorname*{span} \left\{e^x\right\} = \operatorname*{span} \left\{x^k, e^x \mid k = 0,1,\ldots,p\right\} \, ,
\]
has dimension $p+2$ and therefore requires only $p+2$ nodes to satisfy the accuracy conditions of Definition \ref{defn:fsbp-operator}. The quadrature rule, however, has basis
\[
\left( \fnc{F}_p \fnc{F}_p \right)' = \operatorname*{span} \left\{x^k, x^l e^x , e^{2x} \mid k = 0,1,\ldots,2p-1, \ \ l = 0,1,\ldots,p \right\} \, ,
\]
which has dimension $3p+2$, and therefore by Theorem \ref{thm:generalized_gaussian_quadrature} requires $\lceil 3p/2 \rceil + 1$ nodes on an open nodal distribution, $\lceil 3(p+1)/2 \rceil $ nodes on a half-open nodal distribution, or $\lceil 3p/2 \rceil + 2$ nodes on a closed nodal distribution. Effectively, for $p \geq 3$ the SBP operators contain additional DOF once the accuracy conditions and SBP property have been enforced. For other bases, this crossover point can occur at even lower dimensions, especially if the space is not closed under differentiation. Here we provide two methodologies to exploit these additional DOF and obtain optimized SBP operators. Both implementations are publicly available as part of the FSBP-operator construction Julia package \url{https://github.com/alexbercik/GaussFSBP}.

\subsection{A Sequential Optimized Construction Procedure} \label{sec:sequential}

We begin by introducing a sequential optimization procedure, whereupon one first exploits DOF in the extrapolation operators $\vec{t}_L$ and $\vec{t}_R$, then uses the remaining DOF in $\mat{S}$ to satisfy additional objectives. This sequential optimization procedure is commonly employed in the literature~\cite{Marchildon_optimization_2020,mattalo_multidimensional_2019}.

SBP operators will have free parameters if $N>d$, or equivalently, if $\mat{V}^\top$ has a nontrivial nullspace. Therefore, we let $\mat{Z} \in \mathbb{R}^{N \times (N-d)}$ be a matrix with columns that form a basis for $\operatorname*{null}{\left(\mat{V}^\top\right)}$. We compute $\mat{Z}$ through a singular value decomposition. Then by taking any solution for $\vec{t}_L$, for example \eqref{eq:tL_min_norm}, which we hereafter denote by $\vec{t}_L^\star$, all possible extrapolation vectors can be parametrized by some coefficient vectors $\vec{\theta}_L , \vec{\theta}_R \in \mathbb{R}^{N-d}$ as
\begin{equation} \label{eq:tL_tR_parametrization}
\vec{t}_L(\vec{\theta}_L) = \vec{t}_L^\star + \mat{Z} \vec{\theta}_L \, , \qquad 
\vec{t}_R(\vec{\theta}_R) = \vec{t}_R^\star + \mat{Z} \vec{\theta}_R \, ,
\end{equation}
where $\vec{t}_L(\vec{\theta}_L)$ and $\vec{t}_R(\vec{\theta}_R)$ satisfy \eqref{eq:tL_tR_accuracy} by construction since $\mat{V}^\top \mat{Z} = \vec{0}$. We now introduce the following objective function,
\begin{equation} \label{eq:J_L}
  \fnc{J}_L(\vec{\theta}_L) \coloneqq \alpha_{L,\mathrm{acc}} \frac{\fnc{J}_{L,\text{acc}}(\vec{\theta}_L)}{\fnc{J}_{L,\text{acc}}(\vec{0})} + \alpha_{L,\mathrm{nrm}} \frac{\norm{\vec{t}_L(\vec{\theta}_L)}_\star^2}{\norm{\vec{t}_L(\vec{0})}_\star^2} ,
  \end{equation}
which has an accuracy component and a norm component, each weighted by some nonnegative scalar $\alpha_{L,\mathrm{acc}}$ and $\alpha_{L,\mathrm{nrm}}$, respectively, and where $\star$ indicates the same weighted norm used in \eqref{eq:tL_min_norm}. For all numerical experiments we use the $\mat{H}^{-1}$-weighted norm.  The norm component in \eqref{eq:J_L} acts as a regularization, preventing the selection of large cancelling entries that may fortuitously minimize the accuracy component for a chosen set of test functions. The accuracy and norm components of $\fnc{J}_L$ are normalized by their initial unoptimized values in \eqref{eq:J_L} so that the two objectives can be properly compared (essentially as ratios to their initial values). The accuracy component is given by
\begin{equation} \label{eq:J_L_details}
\fnc{J}_{L,\text{acc}}(\vec{\theta}_L) \coloneqq \sum_{m=1}^M \beta_m \left( \frac{\vec{t}_L(\vec{\theta}_L)^\top g_m(\vec{x}) - g_m(x_L)}{g_m(x_L)} \right)^2 ,
\end{equation}
where $\left\{ g_m \right\}_{m=1}^M$ are some test functions that lie outside of $\fnc{F}$, and $\beta_m$ are nonnegative weights for the various $g_m$. In practice, we project $g_m$ to be $\mat{H}$-orthogonal to $\fnc{F}$ to ensure proper conditioning and fair weighting of the different $g_m$ functions, and use a guarded normalization for cases where $g_m(x_L) \approx 0$. Since \eqref{eq:J_L} is a sum of squares, each entry of which is affine in $\vec{\theta}_L$, the minimization of $\fnc{J}_L(\vec{\theta}_L)$ can be recast as the following problem:
\[
  \vec{\theta}_L = \operatorname*{argmin}_{\vec{\theta} \in \mathbb{R}^{N-d}} \norm{\mat{A}_L \vec{\theta} - \vec{b}_L} \, , \qquad 
  \mat{A}_L \in \mathbb{R}^{(M+N) \times (N-d)} \, , \quad
  \vec{b}_L \in \mathbb{R}^{(M+N)} \, ,
\]
which can be solved efficiently for the $N-d$ free variables in a single step using a linear least squares solver. An analogous optimization is performed for $\vec{t}_R$ as well. If endpoints are included, the optimization is skipped in favour of setting $\vec{t}_L = \vec{e}_1$ and $\vec{t}_R = \vec{e}_N$. 

If reflection symmetry of $\vec{t}_L$ and $\vec{t}_R$ is desired (this is only appropriate when the nodes and quadrature
weights also have reflection symmetry), it can be enforced directly through the reflection permutation operator $\mat{\Pi}$, defined through $\vec{t}_R = \mat{\Pi} \vec{t}_L$. Instead of \eqref{eq:tL_tR_accuracy}, the extrapolation exactness constraints are then imposed only on the vector $\vec{t}_L$ through
\[
\begin{bmatrix}
  \mat{V}^\top \\
  \mat{V}^\top \mat{\Pi} 
\end{bmatrix} \vec{t}_L = \begin{bmatrix}
  \vec{v}_L \\
  \vec{v}_R
\end{bmatrix} \, ,
\]
and the space of all possible solutions is correspondingly parametrized as
\[
\vec{t}_L (\vec{\theta}) = \vec{t}_L^\star + \mat{Z}_\mathrm{sym} \vec{\theta} \, , \qquad 
\operatorname*{col}\left(\mat{Z}_\mathrm{sym}\right) = \operatorname*{null}\left( \begin{bmatrix}
  \mat{V}^\top \\
  \mat{V}^\top \mat{\Pi} 
\end{bmatrix} \right) \, , \qquad 
\vec{t}_R (\vec{\theta}) = \mat{\Pi} \vec{t}_L (\vec{\theta}) \, .
\]
The optimization process then proceeds as before.

Any remaining free parameters in $\mat{S}$ can then be optimized in a similar fashion. We begin by reducing the number of equations in \eqref{eq:S_accuracy} by following the procedure of \cite{Marchildon_optimization_2020}. That is, we left-multiply \eqref{eq:S_accuracy} by $\left[ \mat{V} , \mat{Z} \right]^\top$ and separate it into two separate equations. Using the property
\[
\mat{V}^\top \mat{H} \mat{V}' + \left( \mat{V}' \right)^\top \mat{H} \mat{V} = \mat{V}^\top \mat{E} \mat{V} \, ,
\]
we obtain
\begin{align} \label{eq:S_accuracy_andre1}
  \mat{V}^\top \mat{S} \mat{V} &= \tfrac{1}{2} \left( \mat{V}^\top \mat{H} \mat{V}' - \left( \mat{V}' \right)^\top \mat{H} \mat{V} \right) \, , \\
  \mat{Z}^\top \mat{S} \mat{V} &=  \mat{Z}^\top \mat{H} \mat{V}' - \tfrac{1}{2} \mat{Z}^\top \mat{E} \mat{V} \, . \label{eq:S_accuracy_andre2}
\end{align}
Since \eqref{eq:S_accuracy_andre1} is skew-symmetric, it only contributes $d(d-1)/2$ independent equations from its strictly lower triangular entries. As a result, there are a total of $N_e \coloneqq d  \left( 2N - d - 1 \right)/2$ independent equations in (\ref{eq:S_accuracy_andre1}-\ref{eq:S_accuracy_andre2}), and since there are $N_s \coloneqq N(N-1)/2$ unknowns in $\mat{S}$, we identify $N_f \coloneqq (N-d)(N-d-1)/2$ free parameters, which is positive whenever $N>d+1$~\cite{Marchildon_optimization_2020}. 
To optimize over the entries in $\mat{S}$, we first reorder (\ref{eq:S_accuracy_andre1}-\ref{eq:S_accuracy_andre2}) into a linear system of the form 
\begin{equation} \label{eq:S_accuracy_flat}
\mat{A}_s \vec{s} = \vec{b}_s \, , \qquad 
\vec{s} \in \mathbb{R}^{N_s} \, , \quad 
\vec{b}_s \in \mathbb{R}^{N_e} \, , \quad 
\mat{A}_s \in \mathbb{R}^{N_e \times N_s} \, ,
\end{equation}
where $\vec{s}$ contains the strictly lower triangular portion of $\mat{S}$. We could at this point introduce an objective function and employ an optimization algorithm on $\vec{s}$, as in \cite{Marchildon_optimization_2020}. Since here we work exclusively with one dimensional operators, however, the total number of parameters is modest enough such that parametrizing the low-dimensional feasible solution space via the nullspace of $\mat{A}_s$ is computationally feasible. Let 
\[
\vec{s}(\vec{\theta}_s) = \vec{s}^\star + \mat{Z}_s \vec{\theta}_s \, , \qquad
\mat{Z}_s \in \mathbb{R}^{N_s \times N_f} \, , \quad 
\vec{\theta}_s \in \mathbb{R}^{N_f} \, ,
\]
where the columns of $\mat{Z}_s$ form a basis for $\operatorname*{null}{\left(\mat{A}_s\right)}$ and $\vec{s}^\star$ is some solution to \eqref{eq:S_accuracy_flat} (e.g. the minimum norm solution via the Moore-Penrose pseudoinverse). 
We now introduce the following objective function,
\begin{equation} \label{eq:J_s}
\fnc{J}_s(\vec{\theta}_s) \coloneqq \alpha_{s,\mathrm{acc}} \frac{\fnc{J}_{s,\text{acc}}(\vec{\theta}_s)}{\fnc{J}_{s,\text{acc}}(\vec{0})} 
+ \alpha_{s,\mathrm{nrm}} \frac{\norm{\mat{D}(\vec{\theta}_s) + \mat{H}^{-1} \vec{t}_L \vec{t}_L^\top}_F^2}{\norm{\mat{D}(\vec{0}) + \mat{H}^{-1} \vec{t}_L \vec{t}_L^\top}_F^2} \, ,
\end{equation}
where $\alpha_{s,\mathrm{acc}}$, $\alpha_{s,\mathrm{nrm}}$, are again some nonnegative weights, and the norm component (or regularization term) uses the Frobenius norm to additionally bound the spectral radii of an upwind linear convection discretization~\cite{Marchildon_optimization_2020}. The accuracy component is
\begin{equation} \label{eq:J_s_details}
    \fnc{J}_{s,\text{acc}}(\vec{\theta}_s) \coloneqq \sum_{m=1}^M \beta_m \frac{\norm{\mat{D}(\vec{\theta}_s) g_m(\vec{x}) - g_m'(\vec{x})}_*^2}{\norm{g_m'(\vec{x})}_*^2} \, ,
    \end{equation}
where $\left\{ g_m \right\}_{m=1}^M$ and $\beta_m$ are the same $\mat{H}$-orthogonal test functions and weights used in \eqref{eq:J_L_details}, and the subscript $*$ indicates some suitable norm. In practice we use the $\mat{H}$-norm. As before, since $\fnc{J}_s(\vec{\theta}_s)$ is of the form of a sum of squares, each entry of which is affine in $\vec{\theta}_s$, the minimization of $\fnc{J}_s(\vec{\theta}_s)$ can be recast as a linear least-squares minimization problem in $\vec{\theta}_s$, which can be solved efficiently in a single step for the unique $N_f$ optimal free variables.

\subsection{A Simultaneous Optimized Construction Procedure} \label{sec:simultaneous}

The main advantage of the previous sequential optimization is the ability to recast the optimizations as linear least-squares problems in a small-dimensional feasibility space. However, it also has the disadvantage that the initial optimization over $\vec{t}_L$ and $\vec{t}_R$ can remove many DOF that could have been exploited to optimize over $\mat{S}$. For example, when $N = d+1$, we expect to have at least one DOF with which to optimize our operators, but this DOF is removed by the initial boundary operator optimization. Here we introduce an alternative optimization procedure that simultaneously optimizes over all the free parameters in $\vec{t}_L$, $\vec{t}_R$, and $\mat{S}$.

We begin by parametrizing the solution space of $\vec{t}_L$ and $\vec{t}_R$ as before in \eqref{eq:tL_tR_parametrization}. Reflection symmetry can also be enforced similarly to the previous section if desired. The accuracy conditions of Definition \ref{defn:fsbp-operator} are again reformulated using the approach of \cite{Marchildon_optimization_2020}. While \eqref{eq:S_accuracy_andre1} remains identical, \eqref{eq:S_accuracy_andre2} now contains explicit dependence on $\vec{\theta}_L$ and $\vec{\theta}_R$ as
\begin{equation} \label{eq:S_accuracy_sim2}
  \mat{Z}^\top \mat{S} \mat{V} =  \mat{Z}^\top \mat{H} \mat{V}' - \tfrac{1}{2} \mat{Z}^\top \vec{t}_R^\star \vec{v}_R^\top + \tfrac{1}{2} \mat{Z}^\top \vec{t}_L^\star \vec{v}_L^\top
  - \tfrac{1}{2} \mat{Z}^\top \mat{Z} \vec{\theta}_R \vec{v}_R^\top + \tfrac{1}{2} \mat{Z}^\top \mat{Z} \vec{\theta}_L \vec{v}_L^\top \, .
\end{equation}
Let $\vec{s}$ be the entries of $\mat{S}$ as defined before. Since both \eqref{eq:S_accuracy_andre1} and \eqref{eq:S_accuracy_sim2} are affine in $\vec{s}$, $\vec{\theta}_L$ and $\vec{\theta}_R$, we can assemble the accuracy constraints (which also directly encode the SBP property) into a linear system of the form 
\begin{equation} \label{eq:S_accuracy_sim_flat}
\mat{A}_\phi \vec{\phi} = \vec{b}_\phi \, , \qquad
 \vec{\phi} \coloneqq \begin{bmatrix}
  \vec{\theta}_L \\
  \vec{\theta}_R \\
  \vec{s}
\end{bmatrix} \in \mathbb{R}^{2(N-d) + N_s} \, , \quad
\vec{b}_\phi \in \mathbb{R}^{N_e} \, , \quad
\mat{A}_\phi \in \mathbb{R}^{N_e \times (2(N-d) + N_s)} \, .
\end{equation}
Once again we parametrize the low-dimensional feasible solution space via the nullspace of $\mat{A}_\phi$. Let 
\[
\vec{\phi}(\vec{\theta}_\phi) = \vec{\phi}^\star + \mat{Z}_\phi \vec{\theta}_\phi \, , \qquad
\mat{Z}_\phi \in \mathbb{R}^{(2(N-d) + N_s) \times N_t} \, , \quad 
\vec{\theta}_\phi \in \mathbb{R}^{N_t} \, ,
\]
where the columns of $\mat{Z}_\phi$ form a basis for $\operatorname*{null}{\left(\mat{A}_\phi\right)}$, $\vec{\phi}^\star$ is a solution to \eqref{eq:S_accuracy_sim_flat}, and $N_t \coloneqq 2(N-d) + N_f = (N-d)(N-d+3)/2$ is the total number of free parameters. The objective function we use to optimize over $\vec{\theta}_\phi$ is then taken as 
\[
\fnc{J}_\phi(\vec{\theta}_\phi) \coloneqq \fnc{J}_L + \fnc{J}_R + \fnc{J}_s \, ,
\]
where $\fnc{J}_L$, $\fnc{J}_R$, and $\fnc{J}_s$ are as defined in \eqref{eq:J_L} and \eqref{eq:J_s} with appropriate weights $\alpha$, and with explicit dependence on $\vec{\theta}_\phi$ in place of $\vec{\theta}_L$ and $\vec{\theta}_s$. $\fnc{J}_\phi$ can then similarly be recast as a sum of squared residuals, but while the constraint manifold remains linear in $\vec{\theta}_\phi$, the residuals are now quadratic in $\vec{\theta}_\phi$. Therefore, rather than obtaining a linear least squares problem, we must employ a nonlinear least squares solver. We choose the Levenberg-Marquardt method~\cite{levenberg_method_1944,marquardt_algorithm_1963} with finite-difference approximations to the Jacobian. The solver is augmented with a robust multi-start approach to ensure that the global minimum is found, but in practice for the operators we considered, we found that the problem was not multi-modal.

\section{Numerical Results} \label{sec:results}

We now employ the quadrature algorithm described in \S\ref{sec:gen_gauss_algo} along with the three operator construction procedures described in \S\ref{sec:optimization} to construct FSBP operators and apply them to the following one-dimensional linear steady problem,
\begin{equation} \label{eq:steady_problem}
  \der[u]{x} = f(x,u) \, , \quad x \in [0,1] \, , \quad u(0) = 1 \, .
\end{equation}
We discretize this equation using the following standard SBP-SAT approach (see e.g.~\cite{DelReyFernandez2014_review}):
\[
\tfrac{1}{h_\kappa}\mat{D} \vec{u}_\kappa = f(\vec{x}_\kappa, \vec{u}_\kappa) + \tfrac{1}{h_\kappa} \mat{H}^{-1} \left[ \vec{t}_R \left( \vec{t}_R^\top \vec{u}_\kappa - \widehat{\vec{u}}_{\kappa,R} \right) - \vec{t}_L \left( \vec{t}_L^\top \vec{u}_\kappa - \widehat{\vec{u}}_{\kappa,L} \right) \right] \, ,
\]
where $\vec{u}_{\kappa}$ is the numerical solution in element $\kappa$ defined at the physical nodes $\vec{x}_{\kappa}$, and $h_{\kappa}$ is the physical size of the element. We have assumed here that the reference element is defined on $\xi \in [0,1]$ and that an affine mapping is used. At interfaces, elements are coupled using the following SATs:
\begin{align*}
\widehat{\vec{u}}_{\kappa,R} = \tfrac{1}{2} \left( \vec{t}_R^\top \vec{u}_{\kappa} + \vec{t}_L^\top \vec{u}_{\kappa+1} \right) + \tfrac{\sigma}{2} \left( \vec{t}_R^\top \vec{u}_{\kappa} - \vec{t}_L^\top \vec{u}_{\kappa+1} \right) \, , \\
\widehat{\vec{u}}_{\kappa,L} = \tfrac{1}{2} \left( \vec{t}_L^\top \vec{u}_{\kappa} + \vec{t}_R^\top \vec{u}_{\kappa-1} \right) + \tfrac{\sigma}{2} \left( \vec{t}_R^\top \vec{u}_{\kappa-1} - \vec{t}_L^\top \vec{u}_{\kappa} \right) \, ,
\end{align*}
where $\sigma=0$ leads to a symmetric SAT, and $\sigma=1$ to an upwind SAT~\cite{DelReyFernandez2014_review}. At the left boundary, we replace $\vec{t}_R^\top \vec{u}_{0}$ with the left-boundary condition $u_0$ and set $\sigma=1$. No boundary condition is enforced at the right boundary. Once the linear system is assembled, we solve for the numerical solution using a direct linear solve. All the following results can be reproduced using the scripts provided in the open-source repository \url{https://github.com/alexbercik/Paper-GaussFSBP}.

\subsection{Exponential Problem} \label{sec:results_exp}

We first consider $f(x,u) = 8 x u$ such that the exact steady solution is ${u(x) = e^{4 x^2}}$. We use the exponential-enriched polynomial space 
\begin{equation} \label{eq:exp_enriched_space}
\fnc{F}_p \coloneqq \fnc{P}_{p}(\xi) + \operatorname*{span} \left\{e^{a \xi}\right\} = \operatorname*{span} \left\{\xi^k, e^{a \xi} \mid k = 0,1,\ldots,p\right\}
\end{equation}
for some mesh-dependent constant $a$ and reference coordinate $\xi \in [0,1]$. Note that the exact solution is deliberately not exactly represented by our enriched basis for any fixed $a$, thereby mimicking a potential scenario in which only partial solution structure is known. As explained in \S\ref{sec:optimization_need}, we consider $p=3$ and $p=4$ operators so that there is an opportunity for operator optimization. For open operators, there are two free parameters for both $p=3$ and $p=4$, while for closed operators there is only one. The exponent $a$ is chosen so that the enrichment best approximates the component of $e^{4x^2}$ not captured by $\fnc{P}_p$. After projecting out $\fnc{P}_p$, we find that $e^{\alpha x}$ with $\alpha = 11$ leads to an effective $L_2$ approximation on any given physical element, which corresponds to the local basis function exponent $a=\alpha h$ on a uniform mesh with $N_e$ elements of width $h = 1/N_e$.

\begin{remark}
  Selecting an appropriate exponent $\alpha$ in the basis function $e^{\alpha x}$ for the approximation of $e^{4x^2}$ is crucial to ensure optimal performance of the FSBP scheme. For example, for this problem, reducing from $\alpha = 11$ to $\alpha = 6$ reduces the performance of the enriched operators by roughly one order of magnitude (though they still outperform the polynomial operators). In \cite{yuan_discontinuous_2006}, an adaptive projection procedure was employed to find the best possible $\alpha$ for a given numerical solution, which may provide guidance on how to appropriately set the exponent in situations where little is known about the exact solution.
\end{remark}

In Figures \ref{fig:problem1_up} and \ref{fig:problem1_sym} we plot the numerical solution errors in the $\mat{H}$-norm as a function of total DOF. We include results for LG and LGL polynomial operators of degree $p$ and $p+1$, as well as the minimum-norm\footnote{We use the label ``minimum-norm'' to refer to the operators resulting from the standard construction procedure of \S\ref{sec:std_construction}, i.e.\,the minimum Euclidean-norm solution of the linear system defining the entries of $\mat{S}$ in \eqref{eq:S_accuracy}. In general, these are not the same as the operators that minimize the norm objectives found in either of the optimized approaches. When the minimum-norm operators are closed, they match the operators of~\cite{hale_sbp_2026}.} exponential-enriched operators that employ the standard construction procedure of \S\ref{sec:std_construction}. For open operators, we include both the sequential and simultaneous optimized exponential-enriched operators described in \S\ref{sec:sequential} and \S\ref{sec:simultaneous}, respectively. We use the optimization parameters
\begin{equation} \label{eq:optimization_parameters}
\alpha_{L,\mathrm{acc}} = \alpha_{R,\mathrm{acc}} = \alpha_{s,\mathrm{acc}} = 1 \, , \quad 
\alpha_{L,\mathrm{nrm}} = \alpha_{R,\mathrm{nrm}} = \alpha_{s,\mathrm{nrm}} = 0.1 \, ,
\end{equation}
and optimization test functions
\begin{equation} \label{eq:test_functions}
\left\{ g_m(\xi) \right\} = \left\{ \xi^{p+1} , \xi e^{a \xi}, \xi^2 e^{a \xi}, e^{2 a \xi}  \right\} \, , \quad
\left\{ \beta_m \right\} = \left\{ 1, 1, 1, 4 \right\} \, .
\end{equation}
For closed operators, there is no distinction between the sequential and simultaneous optimizations. Note that the degree $p+1$ polynomial operators possess the same number of basis functions as the enriched FSBP operators. Finally, we compare the closed operators to the equispaced FSBP operators of~\cite{Glaubitz_optimization_2025}, which we build using the open-source repository \texttt{SummationByPartsOperatorsExtra.jl}\footnote{\url{https://github.com/JoshuaLampert/SummationByPartsOperatorsExtra.jl}}. These operators employ an optimization procedure to find the $\mat{H}$ quadrature and $\mat{S}$ skew-symmetric matrices simultaneously, and we select operators with the minimal possible number of nodes for a given mesh such that the accuracy conditions of Definition~\ref{defn:fsbp-operator} are satisfied to at least $10^{-8}$ at each node on the reference element.

\begin{figure}[t]
  \centering
  \captionsetup[subfigure]{justification=centering,singlelinecheck=false}

  \begin{subfigure}[t]{0.48\textwidth}
      \centering
      \includegraphics[height=4.5cm, trim={0 10 6 6}, clip]{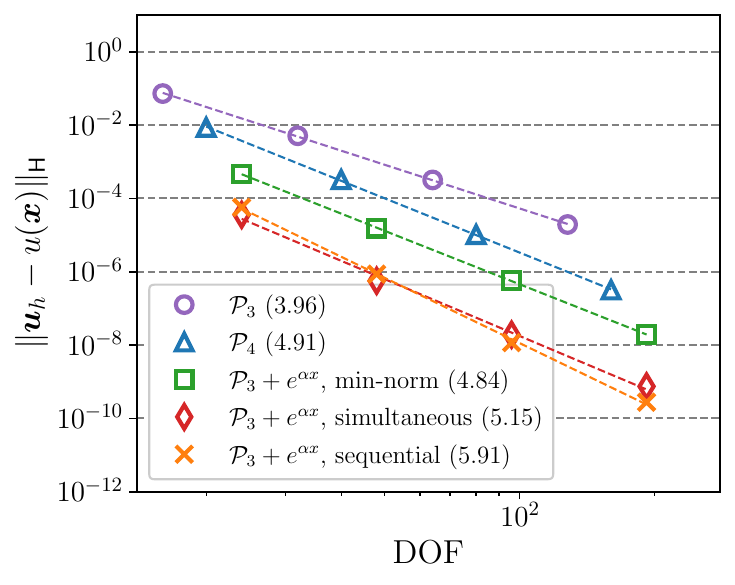}
      \caption{Open, $p=3$.}
  \end{subfigure}
  \hfill
  \begin{subfigure}[t]{0.48\textwidth}
    \centering
    \includegraphics[height=4.5cm, trim={0 10 6 6}, clip]{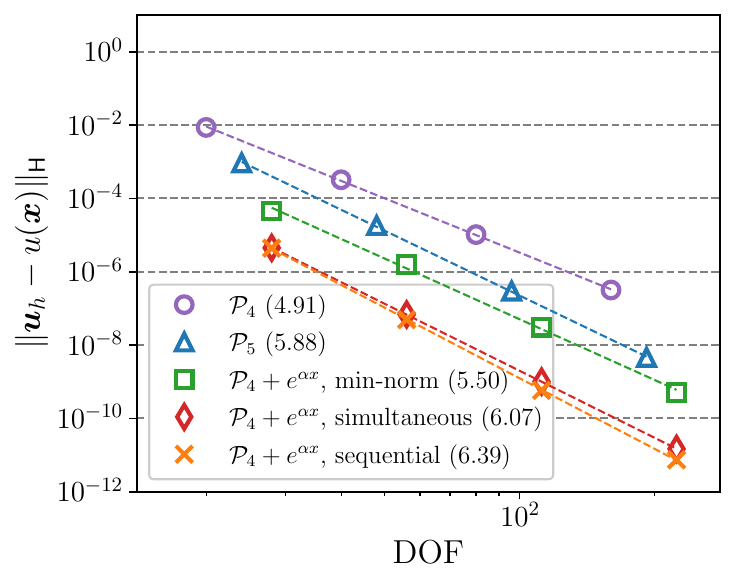}
    \caption{Open, $p=4$.}
\end{subfigure}

  \vspace{0.5em}

  \begin{subfigure}[t]{0.48\textwidth}
    \centering
    \includegraphics[height=4.5cm, trim={0 10 6 6}, clip]{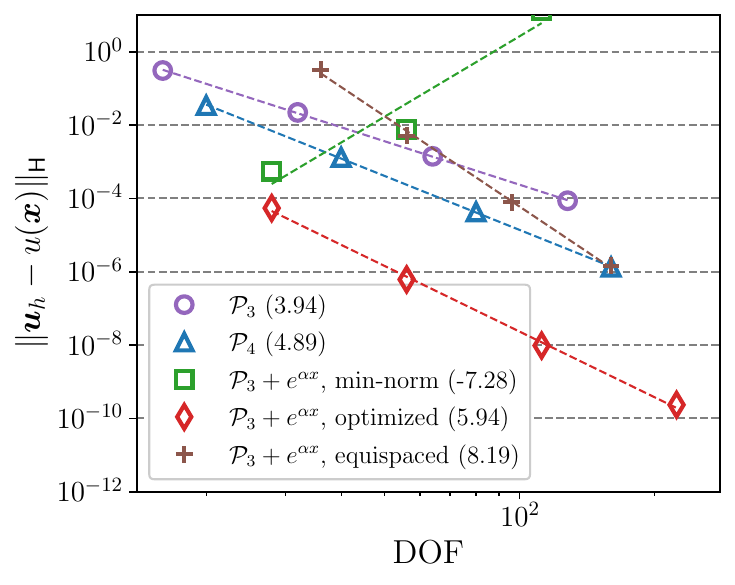}
    \caption{Closed, $p=3$.}
\end{subfigure}
\hfill
\begin{subfigure}[t]{0.48\textwidth}
  \centering
  \includegraphics[height=4.5cm, trim={0 10 6 6}, clip]{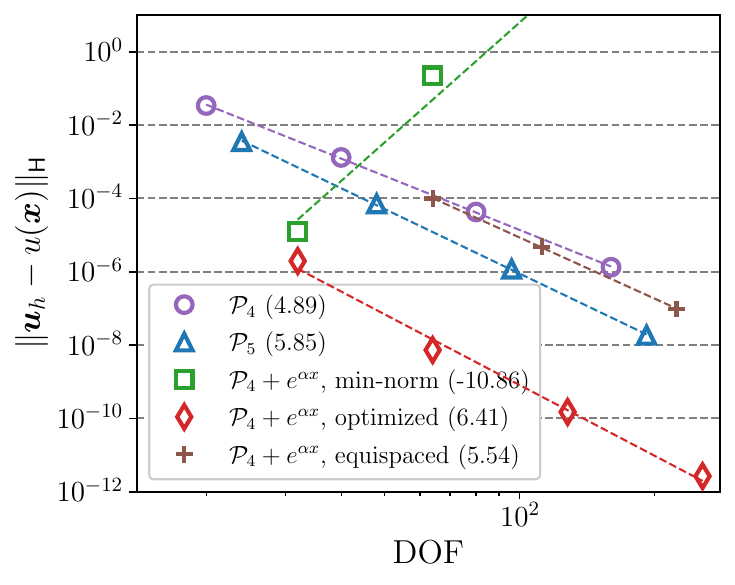}
  \caption{Closed, $p=4$.}
\end{subfigure}

  \caption{Solution errors for the squared exponential problem of \S\ref{sec:results_exp} with upwind SATs. Convergence rates are reported in parentheses in the legend.}
  \label{fig:problem1_up}
\end{figure}

\begin{figure}[t]
  \centering
  \captionsetup[subfigure]{justification=centering,singlelinecheck=false}

  \begin{subfigure}[t]{0.48\textwidth}
      \centering
      \includegraphics[height=4.5cm, trim={0 10 6 6}, clip]{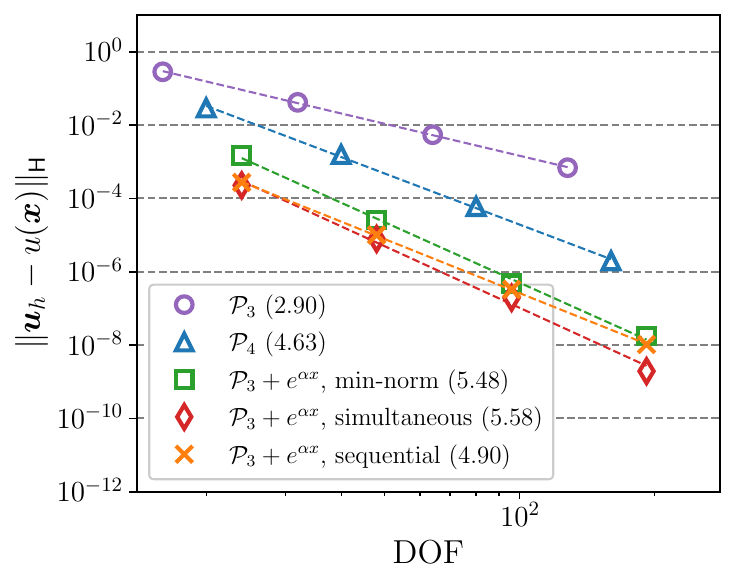}
      \caption{Open, $p=3$.}
  \end{subfigure}
  \hfill
  \begin{subfigure}[t]{0.48\textwidth}
    \centering
    \includegraphics[height=4.5cm, trim={0 10 6 6}, clip]{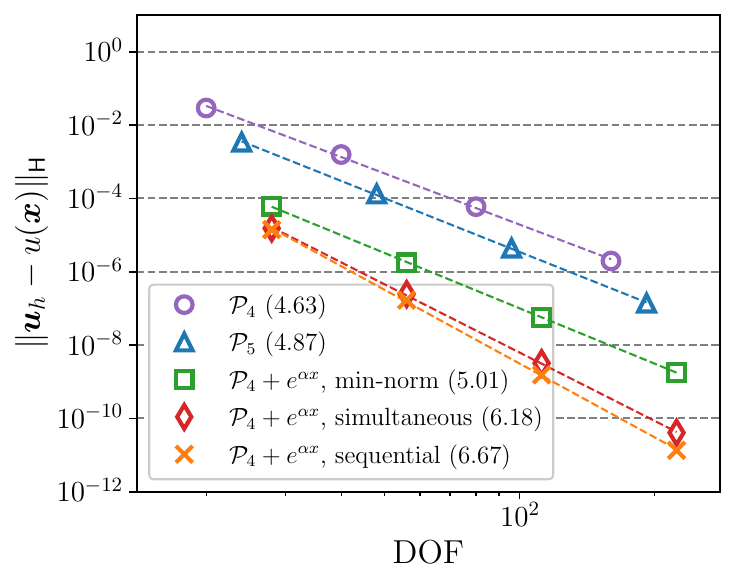}
    \caption{Open, $p=4$.}
\end{subfigure}

  \vspace{0.5em}

  \begin{subfigure}[t]{0.48\textwidth}
    \centering
    \includegraphics[height=4.5cm, trim={0 10 6 6}, clip]{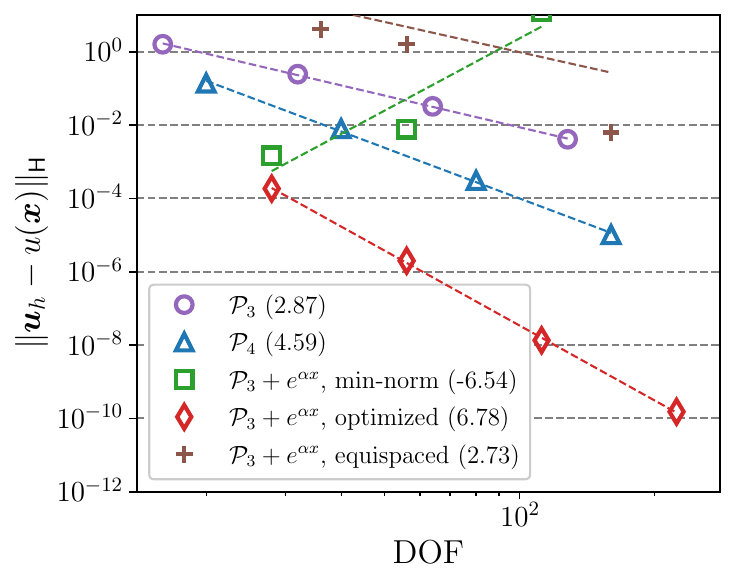}
    \caption{Closed, $p=3$.}
\end{subfigure}
\hfill
\begin{subfigure}[t]{0.48\textwidth}
  \centering
  \includegraphics[height=4.5cm, trim={0 10 6 6}, clip]{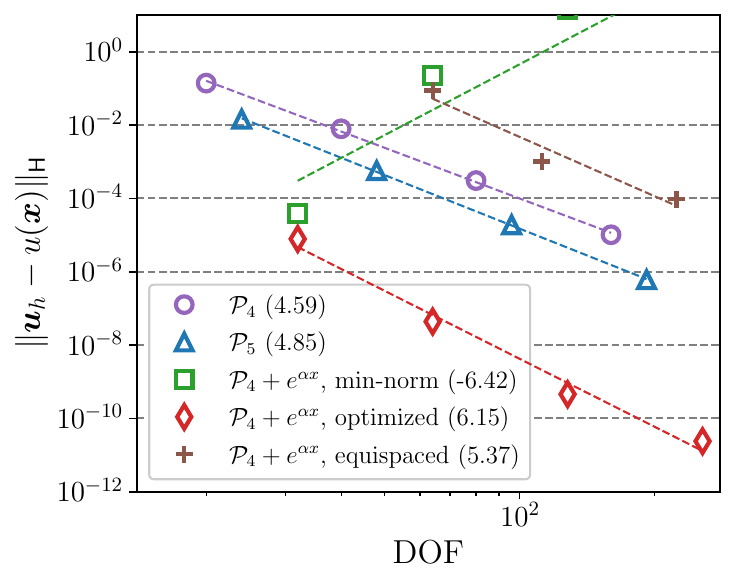}
  \caption{Closed, $p=4$.}
\end{subfigure}

  \caption{Solution errors for the squared exponential problem of \S\ref{sec:results_exp} with symmetric SATs. Convergence rates are reported in parentheses in the legend.}
  \label{fig:problem1_sym}
\end{figure}

Compared to polynomial operators, the exponential-enriched FSBP operators lead to significantly more accurate numerical solutions. Even without any operator optimization, the minimum-norm operators lead to between one and two orders of magnitude smaller error for a given DOF. However, this result is still suboptimal, especially for the closed operators, where the numerical error of the minimum-norm operators is actually shown to increase significantly as the mesh is refined. The reason is two-fold. First, in all cases tested the minimum-norm operator $\mat{D}$ is nullspace-inconsistent (i.e.\,it has a rank of $N-2$ instead of the expected $N-1$), which is well-known to impact the accuracy and convergence of a discretization~\cite{svard_convergence_2019,glaubitz_sbp_not_enough}. Second, for the closed operators, $\mat{D}$ is extremely ill-conditioned, having both very large and very small entries. The operator optimizations resolve both issues and dramatically improve the solution error by an additional one to two orders of magnitude. A similar result was found in \cite{glaubitz_sbp_not_enough}, where adding a regularization term to the optimization construction procedure of equispaced FSBP operators led to more accurate and nullspace-consistent operators. The two optimization procedures (sequential and simultaneous) are found to produce similarly-performing operators, indicating that the extrapolation and volume accuracy objectives $\fnc{J}_{L/R,\mathrm{acc}}$ and $\fnc{J}_{s,\mathrm{acc}}$ do not in general compete. That is, optimizing an operator for extrapolation accuracy of $\vec{t}_L$ and $\vec{t}_R$ also leads to an operator that has an accurate derivative $\mat{D}$. 

For the closed operators, the comparison to equispaced FSBP operators is striking. Because many more nodes are required to find an equispaced operator as opposed to one built with generalized Gaussian quadrature for the same function space, their relative performance with respect to DOF suffers greatly. The equispaced operators exhibit between 3 to 7 orders of magnitude higher numerical error than the optimized FSBP operators. In fact, their solution error is often even worse than that of the polynomial LG and LGL operators. For some cases, the error of the equispaced operators initially converges faster than any of the other operators; however this is simply a result of the equispaced operators requiring fewer nodes to represent the slower exponentials $e^{\alpha h \xi}$ on the reference element as the number of elements $N_e = 1/h$ increases. In fact, if $N_e$ is pushed too high, then the basis becomes nearly linearly dependent and the optimization construction procedure of the equispaced operators struggles to converge. In general, all nullspace-consistent discretizations (i.e.\, excluding the minimum-norm operators) converge at the expected rates of at least $N_b + 0.5$, where $N_b$ is the number of basis functions~\cite{yuan_discontinuous_2006,Yang_dg_2016}\footnote{The expected error convergence rate of $N_b + 0.5$ requires the assumption that the operator basis forms an ET-system~\cite{Yang_dg_2016}, which in this case it does.}.

\subsection{Mixed Exponential Problem} \label{sec:results_mixed_exp}

We next seek to evaluate the performance of the novel FSBP operators on a mixed problem that contains both an exponential boundary layer and a nontrivial smooth solution in the interior. This is representative of a practical scenario in which the problem contains significant flow features that are not captured by the enriched basis. We consider
\[
f(x) = \tfrac{1}{2} \! \left( 1 -2 x \right) \sin(5 \pi x) + \tfrac{5}{2} \pi (x - x^2) \cos(5 \pi x) + \frac{\alpha \, e^{\alpha (x-1)}}{1 - e^{-\alpha}} - 1
\]
such that the exact solution is
\[
u(x) = \tfrac{1}{2} \! \left( x - x^2 \right) \sin(5 \pi x) + \frac{e^{\alpha (x-1)}- e^{-\alpha}}{1 - e^{-\alpha}} - x + 1 \, ,
\]
where we set $\alpha = 80$ in the following experiments. We plot $u(x)$ in Figure \ref{fig:problem2_sol}. 

We employ FSBP operators built with the same exponential-enriched polynomial space \eqref{eq:exp_enriched_space} as in \S\ref{sec:results_exp} with either $p=3$ or $p=4$, and local mesh-dependent constant $a=\alpha h$, where $h = 1/N_e$ is the physical element width. Because the sequential and simultaneous optimization procedures were found to produce similarly-performing exponential-enriched operators in~\S\ref{sec:results_exp}, we only consider the simultaneous-optimized operator for this problem. We use the same optimization parameters \eqref{eq:optimization_parameters} and optimization test functions \eqref{eq:test_functions} as in \S\ref{sec:results_exp}, except that now we emphasize accuracy of the polynomial test function by replacing the test weights in~\eqref{eq:test_functions} with~$\{ \beta_m \} = \{ 4, 1, 1, 1 \}$. For comparison, we again include results for polynomial LG and LGL operators, but now of even higher degrees $p+1$ and $p+2$, as the degree $p+1$ polynomial operators have the same number of basis functions, but the degree $p+2$ operators have comparable DOF. We also include the minimum-norm operators to once again illustrate the importance of the optimization procedures in achieving nullspace-consistency, acceptable operator conditioning, and optimal performance.

\begin{figure}[t]
  \centering
  \includegraphics[height=4.5cm, trim={6 10 6 4}, clip]{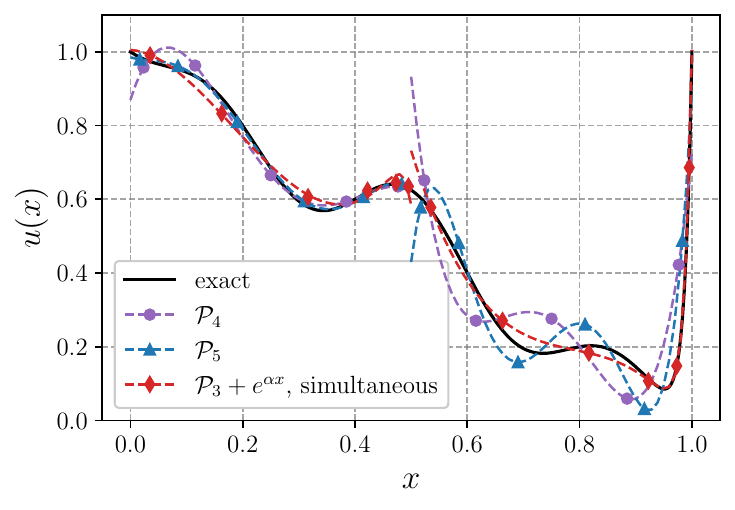}
  \caption{Solution profiles and nodal values for the mixed exponential problem of \S\ref{sec:results_mixed_exp} with upwind SATs, open operators, and two elements. For the polynomial operators, the profiles are interpolations, while for the exponential-enriched operators an $\mat{H}$-weighted projection onto the basis $\fnc{F}_p$ is used (e.g.\,see \cite{Montoya2022}).}
  \label{fig:problem2_sol}
\end{figure}

In Figure \ref{fig:problem2}, we plot the solution errors under mesh refinement with upwind SATs. The optimized FSBP operators again lead to significantly more accurate numerical solutions than the polynomial operators, between four and five orders of magnitude smaller error than the $p+1$ polynomial operators (i.e.\,with the same number of basis functions), and between three and four orders of magnitude smaller error than the higher degree $p+2$ polynomial operators. As seen from the solution profiles in Figure \ref{fig:problem2_sol}, the FSBP operators perform comparably to the polynomial operators on the smooth interior, but are simultaneously able to dramatically reduce the error in the boundary layer, which dominates the numerical error of the polynomial discretizations. Although the minimum-norm operators still outperform the polynomial operators on coarse meshes, for the closed operators, we once again observe that as the basis becomes more linearly dependent, the operators become ill-conditioned and the error increases significantly. Once again, all nullspace-consistent discretizations converge at the expected rates of at least $N_b + 0.5$~\cite{yuan_discontinuous_2006,Yang_dg_2016}, though the optimized FSBP operators consistently exceed this estimate. Similar results are observed for symmetric SATs, demonstrating that the optimized FSBP operators are capable of capturing nontrivial solution features in a manner that is competitive with the polynomial operators, while also resolving known regions of high gradients to a much greater accuracy, thereby achieving overall superior performance.

\begin{figure}[t]
  \centering
  \captionsetup[subfigure]{justification=centering,singlelinecheck=false}

  \begin{subfigure}[t]{0.48\textwidth}
      \centering
      \includegraphics[height=4.5cm, trim={0 10 6 4}, clip]{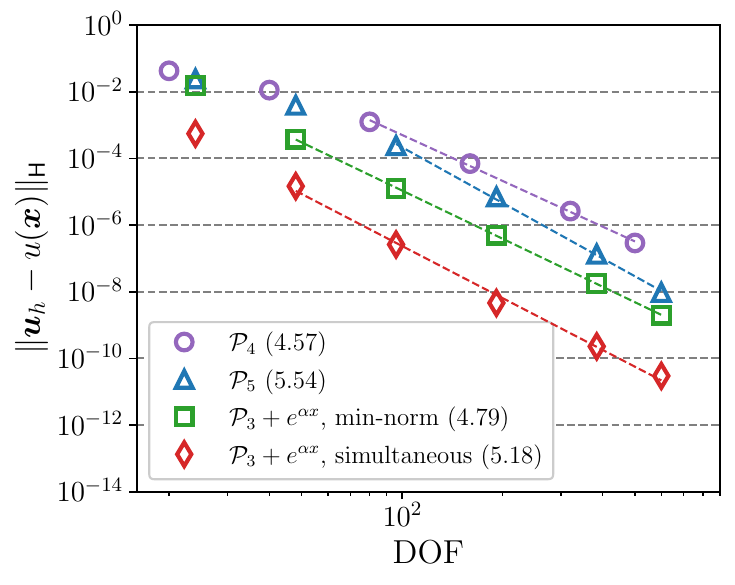}
      \caption{Open, $p=3$.}
  \end{subfigure}
  \hfill
  \begin{subfigure}[t]{0.48\textwidth}
    \centering
    \includegraphics[height=4.5cm, trim={0 10 6 4}, clip]{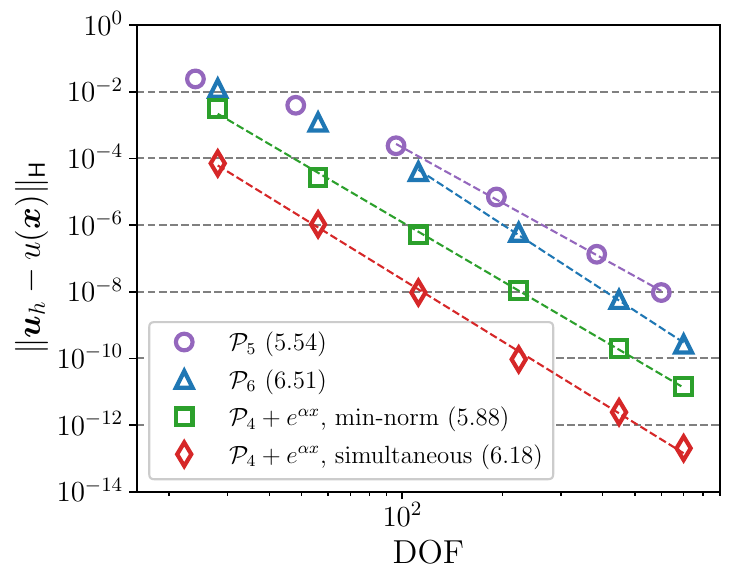}
    \caption{Open, $p=4$.}
\end{subfigure}

  \vspace{0.5em}

  \begin{subfigure}[t]{0.48\textwidth}
    \centering
    \includegraphics[height=4.5cm, trim={0 10 6 4}, clip]{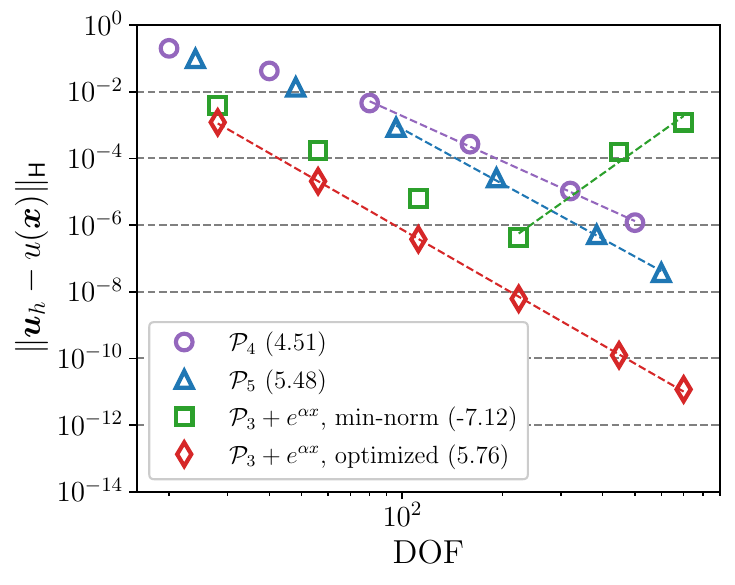}
    \caption{Closed, $p=3$.}
\end{subfigure}
\hfill
\begin{subfigure}[t]{0.48\textwidth}
  \centering
  \includegraphics[height=4.5cm, trim={0 10 6 4}, clip]{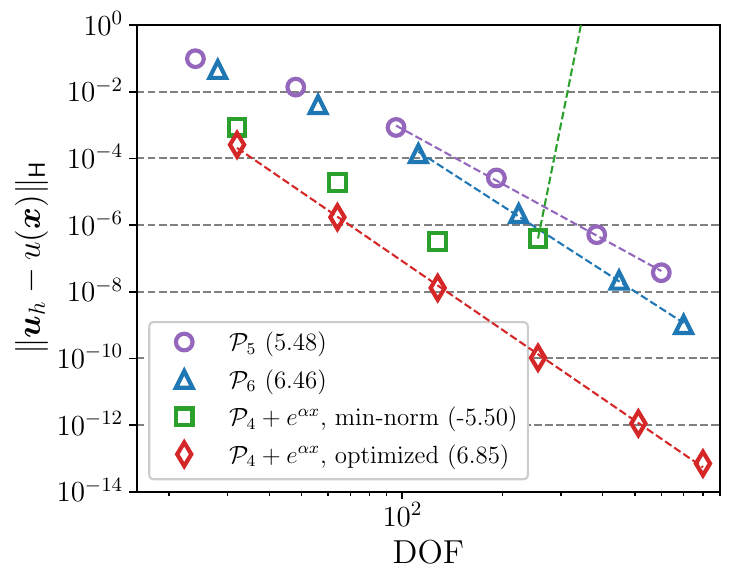}
  \caption{Closed, $p=4$.}
\end{subfigure}

  \caption{Solution errors for the mixed exponential problem of \S\ref{sec:results_mixed_exp} with upwind SATs. Convergence rates are reported in parentheses in the legend.}
  \label{fig:problem2}
\end{figure}

\subsection{Mixed Endpoint Singularity} \label{sec:results_mixed_sqrt}

As a final numerical example, we demonstrate the potential of FSBP operators to significantly reduce numerical error in problems involving singularities. To this end, we introduce the following mixed problem that contains a square-root endpoint singularity combined with the same nontrivial smooth interior as in \S\ref{sec:results_mixed_exp}. We consider
\[
f(x) = \tfrac{1}{2} \! \left( 1 -2 x \right) \sin(5 \pi x) + \tfrac{5}{2} \pi (x - x^2) \cos(5 \pi x) - \frac{1}{\sqrt{1-x}} \, ,
\]
which results in the following exact solution:
\[
u(x) = \tfrac{1}{2} \! \left( x - x^2 \right) \sin(5 \pi x) + 2 \sqrt{1-x}  \, .
\]
We use the shifted square-root-enriched polynomial space 
\begin{equation} \label{eq:sqrt_enriched_space}
\fnc{F}_{p,\ell} \coloneqq \fnc{P}_{p}(\xi) + \operatorname*{span} \left\{\sqrt{\ell-\xi}\right\} = \operatorname*{span} \left\{\xi^k, \sqrt{\ell-\xi} \mid k = 0,1,\ldots,p\right\}
\end{equation}
for some element numbering $\ell$ from the right boundary.
That is, $\ell=1$ denotes the element adjacent to the singular endpoint $x=1$, $\ell=2$ denotes its left neighbour, and so on. The quadrature basis is given by
\begin{equation} \label{eq:sqrt_quadrature_basis}
  \left( \fnc{F}_{p,\ell} \fnc{F}_{p,\ell} \right)' = \operatorname*{span} \left\{
  \xi^k, \frac{\xi^l}{\sqrt{\ell-\xi}}
  \mid k = 0,1,\ldots,2p-1, \
  l = 0,1,\ldots,p\right\} \, ,
\end{equation}
which is of dimension $3p+1$.
The shift in \eqref{eq:sqrt_enriched_space} arises from the portion of the physical function `seen' in the reference element under the affine mapping. On an element $[x_j,x_j+h]$ with $x=h\xi+x_j$ and $x_j=1-h\ell$, we have
\[
\sqrt{1-x} = \sqrt{1-x_j-h\xi}
= \sqrt{h}\sqrt{\ell-\xi}.
\]
The element numbering $\ell$ therefore specifies the location of the singularity in the reference coordinate. In the rightmost element ($\ell=1$), the appropriate basis function is $\sqrt{1-\xi}$ so that the operator can capture the singularity. In all other interior elements ($\ell > 1$), the singularity lies outside the reference interval, resulting in a smooth basis function.

This differs from the scaled exponential enrichment used in 
\S\ref{sec:results_exp}. There we have
\[
e^{\alpha x} = e^{\alpha x_j} e^{\alpha h\xi},
\]
so the element location appears only through a multiplicative constant. Therefore, on a uniform mesh, a single reference operator with exponent $a=\alpha h$ can be used in every element. For the square-root enrichment, however, the element location appears inside the basis function itself, so the reference operator depends on $\ell$, even though the physical size of the element $h$ is irrelevant.

A consequence of using the shifted square-root-enriched basis is that as $\ell$ increases, the basis becomes nearly linearly dependent, making the quadrature algorithm and resulting operators ill-conditioned. Moreover, the enriched operators are primarily useful near the singularity. Therefore, we use the shifted square-root-enriched FSBP operators only in elements with the right-endpoint past $x=0.9$, i.e.\,$x_j + h \geq 0.9$. For elements to the left of $x=0.9$, we replace the shifted square-root basis function with the next degree polynomial basis function, effectively setting 
\[
\fnc{F}(\xi,x_j) = \begin{cases} 
  \fnc{P}_{p+1}(\xi) & \text{if } x_j < 0.9 - h \\
  \fnc{F}_{p,\ell}(\xi) & \text{if } x_j \geq 0.9  - h \end{cases} \, .
\]

\begin{table}[t]
  \centering
  \begin{tabular}{lcccccc}
  Type      & $p$ & 
  $\dim(\fnc{F}_{p,\ell})$ & 
  $\dim((\fnc{F}_{p,\ell} \fnc{F}_{p,\ell})')$ & Add $\xi^{2p}$? & $N$ & $N_f$ \\ \hline
  open & 3 & 5 & 10 & no & 5 & 0 \\
  half-open & 3 & 5 & 10 & yes & 6 & 1 \\
  closed & 3 & 5 & 10 & no & 6 & 0 \\
  open & 4 & 6 & 13 & yes & 7 & 2 \\
  half-open & 4 & 6 & 13 & no & 7 & 1 \\
  closed  & 4 & 6 & 13 & yes & 8 & 1   
  \end{tabular}
  \caption{The dimensions of the operator and quadrature basis functions \eqref{eq:sqrt_enriched_space} and \eqref{eq:sqrt_quadrature_basis}, whether we augment the quadrature basis with $\xi^{2p}$, the number of nodes $N$, and the number of free parameters $N_f$ for the shifted square-root-enriched FSBP operators.}
  \label{tab:sqrt_operators}
\end{table}

We present results for these FSBP operators using both $p=3$ and $p=4$, with either open or closed nodal distributions. For the closed operators, the rightmost element is replaced with a half-open operator to avoid evaluating derivatives at the singular endpoint. The number of nodes, quadrature basis functions, and free parameters available for operator optimization are summarized in Table \ref{tab:sqrt_operators}. When an additional quadrature basis function is required to generate a desired nodal distribution, we augment \eqref{eq:sqrt_quadrature_basis} with the next degree polynomial, $\xi^{2p}$. When free parameters are present in the operator construction, we employ the simultaneous optimization procedure with optimization parameters \eqref{eq:optimization_parameters} and optimization test functions
\begin{equation*}
\left\{ g_m(\xi) \right\} = \left\{ \xi^{p+1} , \xi^{p+2}  \right\} \, , \quad
\left\{ \beta_m \right\} = \left\{ 2, 1 \right\} \, .
\end{equation*}
For comparison, we include results with polynomial LG and LGL operators of degrees $p+1$ and $p+2$, and similarly replace the rightmost LGL operator with a half-open Radau operator. Finally, to emphasize the importance of using the shifted square-root basis \eqref{eq:sqrt_enriched_space}, we also include results for non-shifted square-root basis, which can be obtained by setting $\ell=1$ in \eqref{eq:sqrt_enriched_space} and \eqref{eq:sqrt_quadrature_basis} for all enriched elements. 

\begin{figure}[t]
  \centering
  \captionsetup[subfigure]{justification=centering,singlelinecheck=false}

  \begin{subfigure}[t]{0.52\textwidth}
      \centering
      \includegraphics[height=4.5cm, trim={0 10 6 4}, clip]{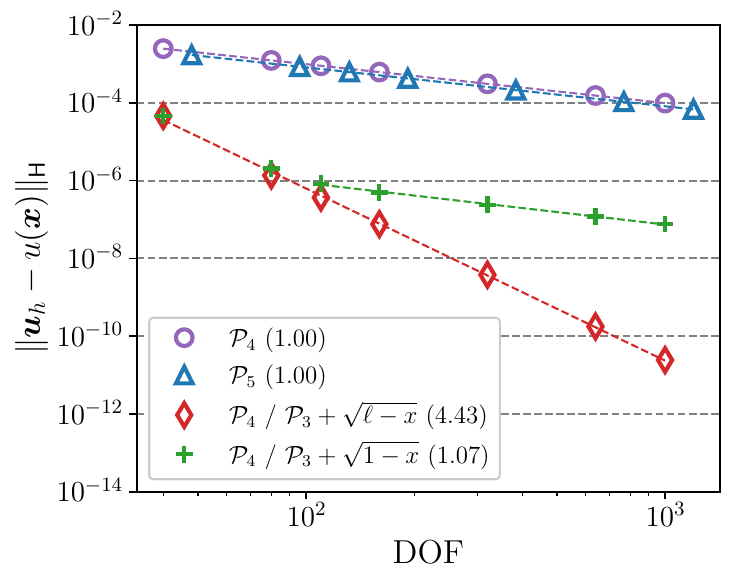}
      \caption{Open, $p=3$.}
  \end{subfigure}
  \hfill
  \begin{subfigure}[t]{0.46\textwidth}
    \centering
    \includegraphics[height=4.5cm, trim={0 10 6 4}, clip]{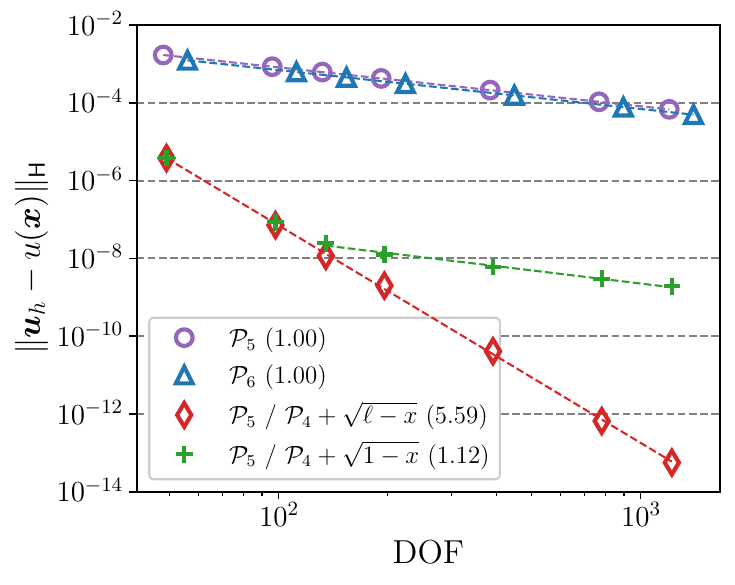}
    \caption{Open, $p=4$.}
\end{subfigure}

  \vspace{0.5em}

  \begin{subfigure}[t]{0.52\textwidth}
    \centering
    \includegraphics[height=4.5cm, trim={0 10 6 4}, clip]{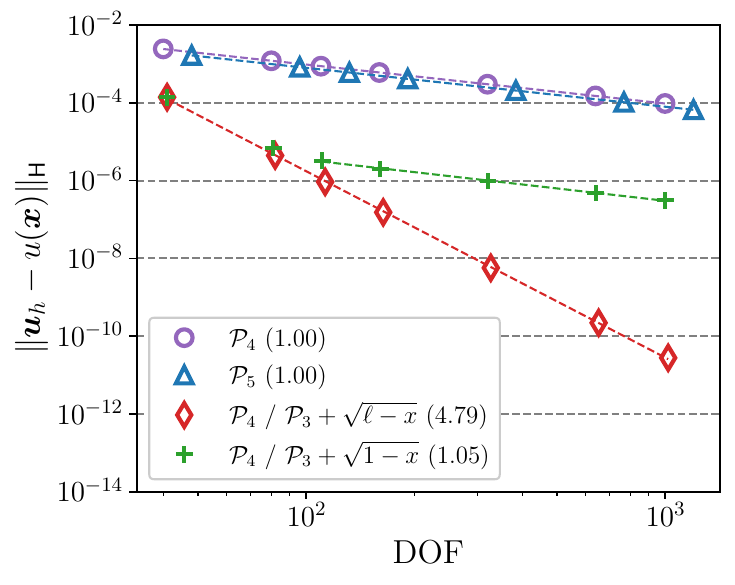}
    \caption{Closed in interior, half-open at right boundary, $p=3$.}
\end{subfigure}
\hfill
\begin{subfigure}[t]{0.46\textwidth}
  \centering
  \includegraphics[height=4.5cm, trim={0 10 6 4}, clip]{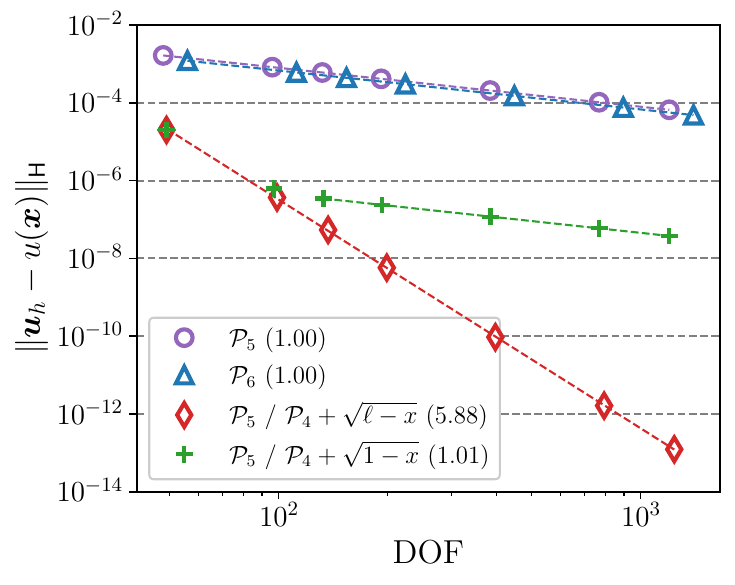}
  \caption{Closed in interior, half-open at right boundary, $p=4$.}
\end{subfigure}

  \caption{Solution errors for the mixed endpoint singularity problem of \S\ref{sec:results_mixed_sqrt} with upwind SATs. Convergence rates are reported in parentheses in the legend.}
  \label{fig:problem3}
\end{figure}

In Figure \ref{fig:problem3}, we compare the solution errors with upwind SATs. Due to the presence of the endpoint singularity, where the derivative is unbounded, the polynomial discretizations converge at a rate of 1, regardless of the degree $p$ of the operator. The shifted square-root-enriched FSBP operators however are able to exactly represent the singularity, and therefore recover design-order convergence of $p+1/2$~\cite{yuan_discontinuous_2006,Yang_dg_2016}, exhibiting drastically smaller errors than the polynomial operators across all mesh refinement levels. The non-shifted square-root basis initially also exhibits much smaller errors than the polynomial operators, but eventually reverts to first-order error convergence as the mesh is refined. This is because the singularity at the right-most element initially dominates the error, so once it is removed exactly in that element, some design-order convergence can be observed by reducing the now-dominant error of the smooth interior. Eventually, however, as the mesh is refined, the elements in the vicinity of the singularity see larger and larger derivatives, and therefore once again dominate the overall error. If the basis function we augment with does not accurately capture the physical function that produces these high gradients (in this case using $\sqrt{1-\xi}$ instead of $\sqrt{\ell-\xi}$), then we no longer see a meaningful benefit of using a non-polynomial basis in those elements. This stresses the importance of tailoring the non-polynomial enrichment functions to the specific problem at hand.

\section{Conclusions and Future Work} \label{sec:conclusion}

We have developed optimized FSBP operators with minimal DOF that, for problems involving large or unbounded gradients where some a priori knowledge of the solution is available, can outperform traditional polynomial operators. Operators with open, closed, and half-open nodal distributions were constructed using generalized Gaussian quadratures found through an improved version of the algorithm of Huybrechs~\cite{Huybrechs_computation_2022}. We provide a publicly available implementation of the modified algorithm that includes a number of improvements that allow the tool to be widely used by the community, including support for non-differentiable bases, arbitrary precision, and robust fallback mechanisms. We also introduced two operator optimization strategies to exploit free parameters when the quadrature rule results in more nodes than needed to satisfy the FSBP operator accuracy conditions. 

We then constructed a few example operators with both an exponential- and a square-root-enriched basis and tested them on problems involving a well-behaved smooth interior solution coupled with a boundary region containing either large or unbounded gradients. In both cases, the optimized FSBP operators outperformed traditional LGL and LG polynomial operators by several orders of magnitude in solution accuracy per DOF. The novel FSBP operators also significantly outperformed equispaced FSBP operators, which in general use significantly more DOF for the same basis. In addition, our novel half-open and open FSBP operators open the door to studying problems with endpoint singularities, where we have demonstrated that if sufficient knowledge of the solution is available, we can recover design-order convergence while other methods are limited to first-order convergence. We also found a clear benefit to the operator optimization strategies introduced in this work, as the standard FSBP construction procedure leads to nullspace-inconsistent and poorly-conditioned operators at higher orders.

Future work should involve the application of these optimized FSBP operators to the study of practical problems, such as those found in aerodynamics (e.g.\,\cite{krank_wall_2018,brill_enriched_2020,zhang_exponential_2020,Fidkowski_basis_2025}). This will require the investigation of curvilinear coordinates, second derivatives, and functional accuracy. We also plan to explore the use of an adaptive basis, such as in \cite{yuan_discontinuous_2006}, for cases where limited solution structure is known. Finally, future work could explore connections to the wider XFEM/GFEM or DEM literature, as incorporating the SBP property into these methods could result in very flexible, efficient, and robust numerical methods suitable for the study of problems with singularities, high gradients, or features requiring high resolution.

\begin{acknowledgements}
  Generative AI (ChatGPT Codex and Cursor Composer) was used to assist in the generation of code for the provided open-source software. The first author is grateful to Zelalem Worku for a helpful discussion regarding the convergence orders of FSBP operators.
\end{acknowledgements}

\section*{Statements and declarations}

\subsection*{Funding}
This research was supported by the Government of Ontario, the University of Toronto, and Ansys, Inc.

\subsection*{Author contributions}

Alex Bercik: Conceptualization, methodology, software, investigation, visualization, writing -- original draft;\\
Lisa Patrascu: Investigation, visualization; \\
David W.\ Zingg: Supervision, conceptualization, writing -- review \& editing.

\subsection*{Conflict of interest}
The authors declare that they have no conflict of interest.

\subsection*{Data Availability}
The scripts to reproduce all of the numerical results are accessible via the open-source repository \url{https://github.com/alexbercik/Paper-GaussFSBP}. The improved quadrature algorithm is available as a Julia package at \url{https://github.com/alexbercik/GeneralizedGauss.jl}, and the operator construction and optimization procedures are also available as a Julia package at \url{https://github.com/alexbercik/GaussFSBP}.

\appendix

\section{The Improved Generalized Gaussian Quadrature Algorithm of Huybrechs} \label{app:gauss_algo}

In this section, we briefly review the generalized Gaussian quadrature algorithm of Huybrechs~\cite{Huybrechs_computation_2022}, together with the modifications we have made to the original implementation. The Julia package is publicly available at \url{https://github.com/alexbercik/GeneralizedGauss.jl}, which itself is a fork of the original implementation available at \url{https://github.com/daanhb/GeneralizedGauss.jl}. For brevity, we will assume the reader is familiar with the theoretical background presented in~\cite{Huybrechs_computation_2022} that is necessary to understand the algorithm.

\paragraph{Overview}

The algorithm is initiated with a vector of callable quadrature basis functions $\left\{ f_j(x) \right\}_{j=1}^K$ that form a CT-system, and optionally, a vector of callable derivatives $\left\{ f_j'(x) \right\}_{j=1}^K$ and moments $\left\{ c_j \right\}_{j=1}^K$ with respect to some desired measure. If the moments are not provided, they are computed to machine precision using the \texttt{BasisFunctions} and \texttt{DomainIntegrals} packages. The code returns different principal representations depending on the parity of $K$.  Let $\ell \in \mathbb{N}$. If $K=2\ell$, then \texttt{principal=:lower} gives the $\ell$-node open rule (e.g. LG), while \texttt{principal=:upper} gives the $(\ell+1)$-node closed rule (e.g. LGL). If $K=2\ell+1$, then \texttt{principal=:lower} gives the $(\ell+1)$-node left half-open rule (e.g. left GR) and \texttt{principal=:upper} gives the $(\ell+1)$-node right half-open rule (e.g. right GR). The description below uses the right-endpoint continuation path (\texttt{add\_endpoint=:right}).  The left-endpoint path is its mirror image: the left endpoint is added with zero weight, the sweep direction is reversed, and the intermediate half-open rules contain $x_L$ instead of $x_R$.

\paragraph{Initialization}

If $K=2$ and the requested rule is the upper principal representation, the code
returns the unique closed two-point rule that solves the system
  \begin{equation} \label{eq:two_point_exact}
  x_1 = x_L, \qquad x_2 = x_R, \qquad \begin{bmatrix} f_1(x_L) & f_1(x_R) \\ f_2(x_L) & f_2(x_R) \end{bmatrix} \begin{bmatrix} \omega_1 \\ \omega_2 \end{bmatrix} = \begin{bmatrix} c_1 \\ c_2 \end{bmatrix} \, .
  \end{equation} 
Otherwise, the continuation starts from the one-point rule exact for
  $\left\{ f_1, f_2 \right\}$.  The following nonlinear scalar problem in $x_1$ is solved,
  \[
      c_1 f_2(x_1) - c_2 f_1(x_1) = 0 \, ,
      \qquad
      \omega_1 = \frac{c_1}{f_1(x_1)} \, ,
  \]
with initial guess $x_1 = (x_L + x_R)/2$. For differentiable bases, this is solved by safeguarded Newton iteration; if Newton fails to make a safe bracketed step, the code falls back to Brent's method.  For the derivative-free path, Brent's method is used directly.

\paragraph{A Full Continuation Cycle}

Assume that the algorithm has already computed an $r$-node lower principal rule (open rule) exact for
$\left\{ f_1,\ldots,f_{2r} \right\}$, which we denote as $\left( \vec{x}^{(2r)}, \vec{\omega}^{(2r)} \right)$. The next cycle consists of four stages.
\begin{description}
    \item[Step 1] Append the right endpoint $x_R$ with weight zero and trace the upper canonical family of $\left\{ f_1,\ldots,f_{2r} \right\}$. That is, we seek upper canonical representations using the nodal distribution $\left\{ x_1 = \xi, x_2(\xi) , \ldots , x_{r+1} = x_R \right\}$ and weights $\left\{ \omega_1(\xi), \ldots , \omega_{r+1}(\xi) \right\}$ parametrized by the continuation parameter $\xi \in \left( x_L,x_1^{(2r)} \right)$, where for example $\omega_i\left( x_1^{(2r)} \right) = \omega_i^{(2r)}$ for $i = 1 , \ldots r$, and $\omega_{r+1}\left( x_1^{(2r)} \right) = 0$. Sweep through the interval $\left( x_L,x_1^{(2r)} \right)$, using a step size of $\Delta x = \left( x_1^{(2r)} - x_L \right) / (n+1)$ with $n=8$. For each $\xi$, obtain the $r+1$ weights and the remaining $r-1$ free nodes of the canonical representation by solving the nonlinear system 
    \begin{equation} \label{eq:upper_canonical_family}
        \sum_{i=1}^{r+1}\omega_i(\xi) f_j(x_i(\xi)) = c_j \, ,
        \qquad j=1,\ldots,2r \, , \quad \text{where} \quad x_{1} = \xi \, , \quad x_{r+1} = x_R \, ,
    \end{equation}
    using the previous canonical representation as an initial guess. Note that on each solve there are a total of $2r$ equations and $2r$ unknowns. For each $\xi$, the next moment residual
    \[
        \Phi_{2r+1}(\xi)
        = \sum_{i=1}^{r+1}\omega_i(\xi) f_{2r+1}(x_i(\xi)) - c_{2r+1}
    \]
    is monitored as $\xi$ moves from $x_1^{(2r)}$ toward $x_L$. Under the CT-system hypotheses, $\Phi_{2r+1}(\xi)$ is monotone, so a sign change brackets the upper principal representation (right half-open rule) for
    $\left\{ f_1,\ldots,f_{2r+1} \right\}$.

    \item[Step 2] Obtain the upper principal representation for $\left\{ f_1,\ldots,f_{2r+1} \right\}$ (i.e.\,the right half-open rule with $r+1$ nodes) by solving the nonlinear system
    \[
        \sum_{i=1}^{r+1}\omega_i f_j(x_i)=c_j,
        \qquad j=1,\ldots,2r+1 \, , \quad \text{where} \quad x_{r+1} = x_R \, ,
    \]
    using the closest canonical representation from the previous step as an initial guess. Note that there are a total of $2r+1$ equations and $2r+1$ unknowns.

    \item[Step 3] Starting from this upper principal representation $\left( \vec{x}^{(2r+1)}, \vec{\omega}^{(2r+1)} \right)$, trace the lower canonical family of $\left\{ f_1,\ldots,f_{2r+1} \right\}$.  We use the nodal distribution $\left\{ x_1 = \xi, x_2(\xi) , \ldots , x_{r+1}(\xi) \right\}$ and weights $\left\{ \omega_1(\xi), \ldots , \omega_{r+1}(\xi) \right\}$ parametrized by the continuation parameter $\xi \in \left( x_L,x_1^{(2r+1)} \right)$, and for each $\xi$, solve an analogous nonlinear system to \eqref{eq:upper_canonical_family}. 
    Note that now only the leftmost node $x_1=\xi$ is fixed;
    the right endpoint $x_{r+1}$ is free. This leads to a system of $2r+1$ equations and $2r+1$ unknowns. For each $\xi$, the monitored residual is
    \[
        \Phi_{2r+2}(\xi)
        = \sum_{i=1}^{r+1}\omega_i(\xi) f_{2r+2}(x_i(\xi)) - c_{2r+2}.
    \]
    A sign change in the monotone $\Phi_{2r+2}(\xi)$ as $\xi$ moves from $x_1^{(2r+1)}$ toward $x_L$ brackets the lower principal representation (open rule) for $\left\{ f_1,\ldots,f_{2r+2} \right\}$.

    \item[Step 4] Obtain the lower principal representation for $\left\{ f_1,\ldots,f_{2r+2} \right\}$ (i.e.\,the open rule with $r+1$ nodes) by solving the nonlinear system
    \[
        \sum_{i=1}^{r+1}\omega_i f_j(x_i)=c_j \, ,
        \qquad j=1,\ldots,2r+2 \, ,
    \]
    using the closest canonical representation from the previous step as an initial guess. Note that there are a total of $2r+2$ equations and $2r+2$ unknowns. This then becomes the seed for the next cycle.
\end{description}

\paragraph{Termination}

The continuation loop stops at the appropriate principal representation. For even $K=2\ell$, selecting \texttt{principal=:lower} stops the loop after the $\ell$-node open rule (Step 4). For odd $K=2\ell+1$, the loop stops after the $\ell+1$-node right half-open rule (Step 2). For even $K=2\ell$ with \texttt{principal=:upper}, the continuation loop terminates after the $\ell$-node right half-open rule (Step 2), and uses it to seed a final `Lobatto' step. In the right-endpoint path, the left endpoint $x_L$ is prepended with weight zero. The upper canonical family of $\left\{ f_1,\ldots,f_{K-1} \right\}$ is then traced using the nodal distribution $\left\{ x_1 = x_L, x_2=\xi , x_3(\xi) , \ldots , x_{K/2+1} = x_R \right\}$ and weights $\left\{ \omega_1(\xi), \ldots , \omega_{K/2+1}(\xi) \right\}$ parametrized by the continuation parameter $\xi \in \left( x_1^{(K-1)}, x_2^{(K-1)} \right)$. For each $\xi$ in the sweep, the nonlinear system
\[
  \sum_{i=1}^{K/2+1}\omega_i(\xi) f_j(x_i(\xi)) = c_j \, ,
  \qquad j=1,\ldots,K-1 \, , \quad \text{where} \quad x_{1} = x_L \, , \quad x_2 = \xi \, , \quad x_{K/2+1} = x_R \, ,
\]
is solved  using the previous canonical representation as an initial guess. The monitored residual is
\[
  \Phi_{K}(\xi)
  = \sum_{i=1}^{K/2+1}\omega_i(\xi) f_{K}(x_i(\xi)) - c_{K} \,
\]
which is again monotone, so a sign change brackets the desired upper principal representation (closed rule) for $\left\{ f_1,\ldots,f_{K} \right\}$. The resulting bracket then seeds the final closed solve
\[
    \sum_{i=1}^{K/2+1}\omega_i f_j(x_i)=c_j \, ,
    \qquad j=1,\ldots,K \, , \quad
    x_1=x_L \, , \quad x_{K/2+1}=x_R \, .
\]

\paragraph{Nonlinear Solves}
When the basis functions $f_j$ are differentiable, the nonlinear solves use a vector residual of the form
\[
  r_j\left( \vec{x}^{(\mathrm{free})},\vec{\omega} \right)
  \coloneqq \sum_{i=1}^{m} \omega_i f_j(x_i) - c_j = 0 \, ,
  \qquad j=1,\ldots,m \, ,
\]
for $n$ nodes and $m$ basis functions. If $k$ nodes are fixed, we always have $2n - k = m$, so there is an equal number of equations and unknowns. For any $x_i$ that is a free node and $\omega_i$, the Jacobian entries are
\[
    \frac{\partial R_j}{\partial x_i}=\omega_i f'_j(x_i) \, , \qquad
    \frac{\partial R_j}{\partial \omega_i}=f_j(x_i) \, .
\]
Analytic first derivatives are used when available; if a set of callable derivatives $f_j'(x)$ are not provided but \texttt{differentiable=true} is set, then the missing first derivatives are approximated with a support-aware finite differences routine. We employ the quadratic trust region method via the package \texttt{NLsolve.jl}.

If \texttt{differentiable=false}, the nonlinear systems are instead solved using a mesh adaptive direct search (MADS) method via the package \texttt{NOMAD.jl}. In this case, we reduce the dimensionality of the problem to only solve for the free nodes $\vec{x}^{(\mathrm{free})}$ by employing a least-squares projection for the weights $\vec{\omega}$,
\[
    \vec{\omega}\left( \vec{x}^{(\mathrm{free})} \right)
    =
    \operatorname*{argmin}_{\vec{\eta}}
    \left\|\mat{V}(\vec{x})\vec{\eta}-\vec{c}\right\|_2 \, ,
    \qquad
    \mat{V}_{j i}(\vec{x}) = f_j(x_i) \, .
\]
Note that $\vec{x}^{(\mathrm{free})} \in \mathbb{R}^{n-k}$, $\vec{x}, \vec{\omega} \in \mathbb{R}^{n}$, and $\vec{c} \in \mathbb{R}^{m}$, so the above projection is overdetermined, meaning the moment constraints will not be satisfied exactly until the appropriate $\vec{x}^{(\mathrm{free})}$ are found. To find $\vec{x}^{(\mathrm{free})}$, the MADS optimizer minimizes the scalar residual
\[
    \left\|\mat{V}\left(\vec{x} \left(\vec{x}^{(\mathrm{free})} \right)\right) \vec{\omega}\left(\vec{x}^{(\mathrm{free})} \right) -\vec{c}\right\|_2^2 .
\]
The current implementation enforces support-bounded ordered
nodal distributions and hard bounds on $\vec{x}^{(\mathrm{free})}$ parameters whenever a canonical bracket is available.

If the nonlinear solver fails during a canonical sweep, the step size is reduced by half and the solver is restarted using the same initial guess, which is now closer to the exact solution. Similarly, if the nonlinear solver fails during a principal representation solve, the bracket for $\xi$ obtained from the previous canonical sweep is used to form a secant update for a closer $\xi$, from which we recompute the canonical representation, resulting in a closer initial guess for the principal representation solve. In practice, however, these fallback mechanisms are rarely needed as the continuation process already results in extremely accurate initial guesses.

For any nonlinear solve, residuals slightly above the strict solver tolerance can be accepted through the \texttt{lost\_digits} policy, which is useful for ill-conditioned bases. Since the various non-terminal solves are only used to obtain initial guesses for the next step in the continuation process, we may also use a looser \texttt{intermediate\_tolerance} on each of these solves to improve efficiency. The final returned quadrature rule, however, is still solved to the full strict solver tolerance. 

\paragraph{A Summary of Modifications}

The above algorithm includes a number of modifications and generalizations from the original implementation of Huybrechs~\cite{Huybrechs_computation_2022}. These include:
\begin{itemize}
  \item a generalized endpoint continuation via \texttt{add\_endpoint=:left}
  or \texttt{:right}. This makes the algorithm suitable for the generation of quadrature rules with either left or right endpoint singularities.
  \item support for quadrature rules with left half-open, right half-open, and closed nodal distributions;
  \item allowing moments to be passed explicitly;
  \item adaptive canonical sweeps with monotonicity checks on $\Phi$, resulting in significant speedup in the continuation process and guaranteed robustness;
  \item secant-based recovery after failed principal solves, resulting in guaranteed robustness in the continuation process;
  \item support-aware finite differences for missing first derivatives and a derivative-free MADS path for non-differentiable bases;
  \item the addition of automatic QR-based orthogonalization of ill-conditioned bases with triangular transformations that preserve the ET-system property;
  \item optional moment computation with a user-supplied measure,
  \item arbitrary precision support via \texttt{BigFloat}; 
  \item numerical T-system and ECT-system diagnostic checks based on generalized Vandermonde determinants and Wronskians (see Appendix \ref{app:chebyshev_tests});
  \item expanded documentation, driver scripts, and regression tests.
\end{itemize}

\section{Methods for Determining whether a Basis is a Chebyshev System} \label{app:chebyshev_tests}

The direct way to verify that a basis $\{ f_j \}_{j=1}^K$ is a T-system is to use Definition~\ref{defn:t-system}, i.e.\, to verify that for every ordered nodal distribution
$x_L \leq x_1 < \cdots < x_K \leq x_R$, the generalized Vandermonde determinant
\[
  \det \left[ f_j(x_i) \right]_{i,j=1}^K
\]
is positive. This is straightforward to turn into a numerical diagnostic: sample many ordered point sets, evaluate the determinant, and check whether the normalized determinant remains bounded away from zero with a constant sign. It is not, however, a practical proof, since the definition requires checking a continuum of possible point sets. The \texttt{check\_T\_system} diagnostic exported by \texttt{GeneralizedGauss.jl} performs this numerical test.

A stronger and often easier condition to verify is that the basis is an ECT-system (Definition~\ref{defn:et-system}). For $C^{K-1}$-continuous ordered basis, the Wronskians
\[
  W_k(x) \coloneqq
  \det \left[ f_j^{(i-1)}(x) \right]_{i,j=1}^k,
  \qquad k=1,\ldots,K .
\]
must have a constant nonzero sign on the interval. In practice this is usually much simpler to inspect, since one can plot the Wronskians as functions of $x$ rather than sample determinants over all ordered point sets. In some cases it may also be convenient to check the product space $\fnc{F}\fnc{F}$ directly rather than $(\fnc{F}\fnc{F})'$, as whenever the first function in $\fnc{F}$ is a constant, the first column of the Wronskian contains exactly one nonzero entry and zeros below it, so expanding along that column gives the Wronskian of the differentiated remaining functions, i.e.\,the corresponding $(\fnc{F}\fnc{F})'$ basis. The \texttt{check\_ECT\_system} diagnostic exported by \texttt{GeneralizedGauss.jl} implements this stronger test by evaluating the initial Wronskians at interior Chebyshev sample points, using analytic derivatives when available and otherwise falling back to Chebyshev interpolation/Taylor approximations. Both diagnostics are wrapped in \texttt{GaussFSBP.jl}.

For some bases, the Chebyshev-system property can be shown directly with little effort. There is no general recipe for such proofs, however. For the exponential-enriched space \eqref{eq:exp_enriched_space}, the quadrature basis is
\[
  (\fnc{F}_p\fnc{F}_p)' =
  \operatorname*{span}\left\{\xi^k,\ \xi^l e^{a\xi},\ e^{2a\xi}
  \mid k=0,\ldots,2p-1,\ l=0,\ldots,p \right\}.
\]
Assume $a \neq 0$ and order the basis as written. The initial polynomial blocks have constant nonzero Wronskians, as the derivative matrices are upper triangular with diagonal entries
\(0!,1!,\ldots,k!\). Let $\partial_\xi$ denote the differential operator $\mathrm{d}/\mathrm{d}\xi$. Then, the mixed initial segment
\[
  \left(1,\xi,\ldots,\xi^{2p-1},e^{a\xi},\xi e^{a\xi},\ldots,\xi^l e^{a\xi}\right)
\]
is a fundamental system for the constant-coefficient operator
$\partial_{\xi}^{2p}(\partial_{\xi}-a)^{l+1}$. This $(2p+l+1)$-order ODE has a characteristic polynomial $\lambda^{2p}(\lambda-a)^{l+1}$ with roots $\lambda=0$ of multiplicity $2p$, and $\lambda=a$ of multiplicity $l+1$. In monic form, this is
\[
\lambda^{2p+l+1} - a(l+1)\lambda^{2p+l} + \cdots = 0 \, ,
\]
so Abel's identity gives
$W(\xi)=C_l e^{(l+1)a\xi}$ for some $C_l\neq 0$, which is clearly of a constant nonzero sign. The full basis is similarly a fundamental system for
$\partial_\xi^{2p}(\partial_\xi-a)^{p+1}(\partial_\xi-2a)$, so we can proceed in the same manner to show that the Wronskian is
$C e^{(p+3)a\xi}$ with $C\neq 0$. We conclude that since all Wronskians have a constant nonzero sign, the exponential quadrature basis is an ECT-system on any finite interval for $\xi$.

For the shifted square-root space \eqref{eq:sqrt_enriched_space}, the quadrature basis \eqref{eq:sqrt_quadrature_basis} is
\[
  (\fnc{F}_{p,\ell}\fnc{F}_{p,\ell})' =
  \operatorname*{span}\left\{
  \xi^k,\ \frac{\xi^l}{\sqrt{\ell-\xi}}
  \mid k=0,\ldots,2p-1,\ l=0,\ldots,p \right\}.
\]
Let $t=\ell-\xi$. On any interval where $t>0$, the polynomial block spans
$\left\{1,t,\ldots,t^{2p-1}\right\}$ and the square-root block spans
$\left\{t^{-1/2},t^{1/2},\ldots,t^{p-1/2}\right\}$. Hence the same space is spanned by monomials $t^{\alpha_j}$ with exponents
\[
  \left\{\alpha_j\right\}_{j=1}^{3p+1}
  =
  \left\{-\tfrac{1}{2},\tfrac{1}{2},\ldots,p-\tfrac{1}{2}\right\}
  \cup
  \left\{0,1,\ldots,2p-1\right\},
\]
ordered increasingly.
Through direct computation, the Wronskian of any initial 
$\left\{t^{\alpha_j}\right\}_{j=1}^k$ is
\[
\begin{aligned}
  W_k(t)
  &= \det \left[ \left( t^{\alpha_j} \right)^{(i-1)} \right]_{i,j=1}^k \\
  &= \det\left[(\alpha_j)_{i-1} t^{\alpha_j-(i-1)}\right]_{i,j=1}^{k} \\
  &=
  \left(\prod_{j=1}^{k} t^{\alpha_j}\right)
  \left(\prod_{i=1}^{k} t^{-(i-1)}\right)
  \det\left[(\alpha_j)_{i-1}\right]_{i,j=1}^{k} \, ,
\end{aligned}
\]
where $(\alpha)_m \coloneqq \alpha(\alpha-1)\cdots(\alpha-m+1)$ is the falling factorial. For each $m$, the falling factorial $(\alpha)_m$ is a monic polynomial in $\alpha$ of degree $m$, so the row vector of falling factorials
\[
  \left(1,\alpha,\alpha(\alpha-1),\ldots,(\alpha)_{k-1}\right)^\top
\]
can be obtained from the monomial vector
\[
  \left(1,\alpha,\alpha^2,\ldots,\alpha^{k-1}\right)^\top
\]
by multiplication with a lower triangular matrix whose diagonal entries are
all one. This change of basis has determinant one, so
\[
  \det\left[(\alpha_j)_{i-1}\right]_{i,j=1}^{k}
  =
  \det\left[\alpha_j^{i-1}\right]_{i,j=1}^{k}
  =
  \prod_{1\leq i<j\leq k}(\alpha_j-\alpha_i)  \neq 0\, ,
\]
where the right-hand side is the well-known Vandermonde determinant of monomials. This is nonzero since each exponent $\alpha_j$ is distinct. Therefore, the full Wronskian of $\left\{t^{\alpha_j}\right\}_{j=1}^k$ is
\[
  W_k(t) =
  \left(\prod_{1\leq i<j\leq k}(\alpha_j-\alpha_i)\right)
  t^{\sum_{j=1}^{k}\alpha_j-k(k-1)/2} \, ,
\]
which is nonzero for $t>0$. Thus the monomial-ordered basis is an ECT-system. The Wronskian
with respect to $\xi$ differs only by the fixed sign
$(-1)^{k(k-1)/2}$, since $\mathrm{d}/\mathrm{d}\xi=-\mathrm{d}/\mathrm{d}t$, and since the original shifted square-root quadrature basis and the monomial basis are related by an invertible linear change of basis, their full generalized Vandermonde determinants differ only by a nonzero constant factor. Therefore the shifted square-root quadrature basis is a T-system. In the case $\ell=1$, this means the result applies on $[0,1)$, with the singular endpoint handled by open or half-open quadrature rules.

\bibliographystyle{spmpsci}      
\bibliography{references.bib}

\end{document}